\documentclass[leqno,11pt,a4paper,twoside]{article}%
\usepackage{amssymb}
\usepackage{amsfonts}
\usepackage{amsmath}
\usepackage{authblk}
\usepackage{graphicx}
\usepackage{ifthen}
\usepackage{comment}
\usepackage{xcolor}
\usepackage{braket}
\usepackage{url}
\usepackage{epsfig}
\usepackage{float}
\usepackage{enumitem}
\usepackage{lipsum,multicol}
\usepackage[T1]{fontenc}
\usepackage[english]{babel}
\usepackage[english]{babel}
\setlist[itemize]{noitemsep, topsep=0pt}
\setlist[enumerate]{noitemsep, topsep=0pt}
\setlist[itemize]{leftmargin=*}
\setlist[enumerate]{leftmargin=*}
\providecommand{\U}[1]{\protect\rule{.1in}{.1in}}

\providecommand{\norm}[1]{\left\lVert#1\right\rVert}
\providecommand{\abs}[1]{\left\lvert#1\right\rvert}
\providecommand{\pr}[1]{\left(#1\right)} 
\providecommand{\pp}[1]{\left[#1\right]} 
\providecommand{\set}[1]{\left\lbrace#1\right\rbrace} 
\providecommand{\scal}[1]{\left\langle#1\right\rangle}

\providecommand{\Db}[1]{\textcolor{black}{#1}}

\newcommand{\normi}[3]{\norm{#1}_{
    \ifthenelse{\equal{#2}{1}}{H_0^1\pr{\mathcal{O}_{#3}}}{%
    \ifthenelse{\equal{#2}{-1}}{H^{-1}\pr{\mathcal{O}_{#3}}}{}}}}

\makeatletter
\newcommand{\subjclass}[2][2020]{%
  \let\@oldtitle\@title%
  \gdef\@title{\@oldtitle\footnotetext{\textbf{#1 \emph{Mathematics subject classification.}} #2}}%
}
\newcommand{\keywords}[1]{%
  \let\@@oldtitle\@title%
  \gdef\@title{\@@oldtitle\footnotetext{\textbf{\emph{Key words.}} #1.}}%
}
\makeatother

\newtheorem{theorem}{Theorem}

\newtheorem{definition}[theorem]{Definition}

\newtheorem{proposition}[theorem]{Proposition}
\newtheorem{remark}[theorem]{Remark}

\newenvironment{proof}[1][Proof]{\noindent\textbf{#1.} }{\ \rule{0.5em}{0.5em}}

\newtheorem{ass}{Assumption}[section]

\author[1,2]{Dan Goreac}
\author[2,3]{Juan Li}
\author[4]{Pangbo Wang}
\affil[1]{\small \'{E}cole d'Actuariat,  Universit\'{e} Laval, Qu\'{e}bec (QC), G1V 0A6, Canada, \textit{Email: dan.goreac@act.ulaval.ca}} 
\affil[2]{\small School of Mathematics and Statistics, Shandong University, Weihai, Weihai 264209, P.R. China}
\affil[3]{\small Research Center for Mathematics and Interdisciplinary Sciences, Shandong University, Qingdao 266237, P. R. China, \textit{Email: juanli@sdu.edu.cn}}
\affil[4]{\small Zhongtai Securities Institute for Financial Studies Shandong University, Jinan 250100, P.R. China, \textit{E-mail: pangbo.wang@mail.sdu.edu.cn}}
\date{\vspace{-10ex}}

\begin{document}
\title{A Relaxed Control Problem With $\mathbb{L}^\infty$ Cost and Jump Dynamics Motivated by Cyber Risks Insurance\footnote{The authors acknowledge financial support from the NSF of Shandong Province (No. ZR202306020015), the NSF of P.R. China (Nos. W2511002, 12031009), and the National Key R and D Program of China (No. 2018YFA0703900).\\
D.G. acknowledges financial support from National Sciences and Engineering Research Council (NSERC), Canada,  Grant/Award Number: RGPIN-2025-03963.}}
\maketitle
\begin{abstract}
This paper has a double aim. One the one hand, we introduce a uni-nodal network model for cyber risks with firewalled edges and SIR intra-edge spreading. In connection to this, we formulate an insurance problem in which one seeks the running maximal reputation index against all control strategies of the companies represented by edges. On the other hand, we seek to characterize the value function with $\mathbb{L}^\infty$ cost through linear programming techniques and more standard Hamilton-Jacobi integro-differential inequalities.
\end{abstract}
\section{Introduction}
{The COVID-19 pandemic and recent political and geopolitical events have increasingly influenced everyday vocabulary, introducing widespread use of terms such as "(cyber-)risks" and "contagion/spreading". Naturally, scientific research reflects this reality, as evidenced by the special issue on "Cyber Risk and Security" \cite{DK_2022} and various research initiatives \cite{HL_2018}.

The concept of \emph{cyber-risk}, derived from IT, is inherently linked to the notion of \emph{network}. Consequently, numerous scientific studies focus on network structures combined with forms of spreading. For example, \cite{FWW_2018} examines cyber-loss pricing in a model where nodes are binary and infection/recovery mechanisms operate with jump rates aggregated through an adjacency matrix. Mathematically, the state vector is discrete, allowing simulations using Gillespie's algorithm \cite{Gillespie1977}. Similar clustering-based models are discussed in \cite{Qi_2023} and \cite{AWS_2021}. Other approaches adopt game theory within supply chain contexts, e.g., \cite{NDS_2017}. Additionally, the Susceptible-Infected-Recovered (SIR) model from \cite{KeMcK}, adapted for portfolio elements, can be incorporated into network-based frameworks \cite{HL_2021}, with interaction matrices quantifying compartmental connections \cite{HL_2022}.

These models generally exclude control measures, whether curative (e.g., internal network fault management) or protective (e.g., external anti-malware interventions). However, epidemic control interventions—such as network access restrictions or physical unit restoration—have been extensively studied. Analogous measures in population models include social distancing and vaccination \footnote{The analogous measures in population epidemic models would be social distancing and vaccination.}. Relevant research includes access restriction models \cite{anderson1992infectious,Behncke,hansen2011optimal} and restoration-type interventions \cite{Biswas_2014,Wang_2019,Acunaetal_2021}. The COVID-19 pandemic has further expanded this literature, with studies addressing access-like restrictions \cite{alvarez2020simple,Kruse,Ketch,bolzoni2019optimal,Fre22,AFG_2022} and restoration interventions \cite{Biswas_2014,Wang_2019,Acunaetal_2021}. Stochastic control models have also gained attention in this context \cite{GLX_2022}.}

We envisage a star-shaped model as in Figure \ref{Fig1}. \begin{figure}[H]
\centerline{\includegraphics[width=0.5\columnwidth]{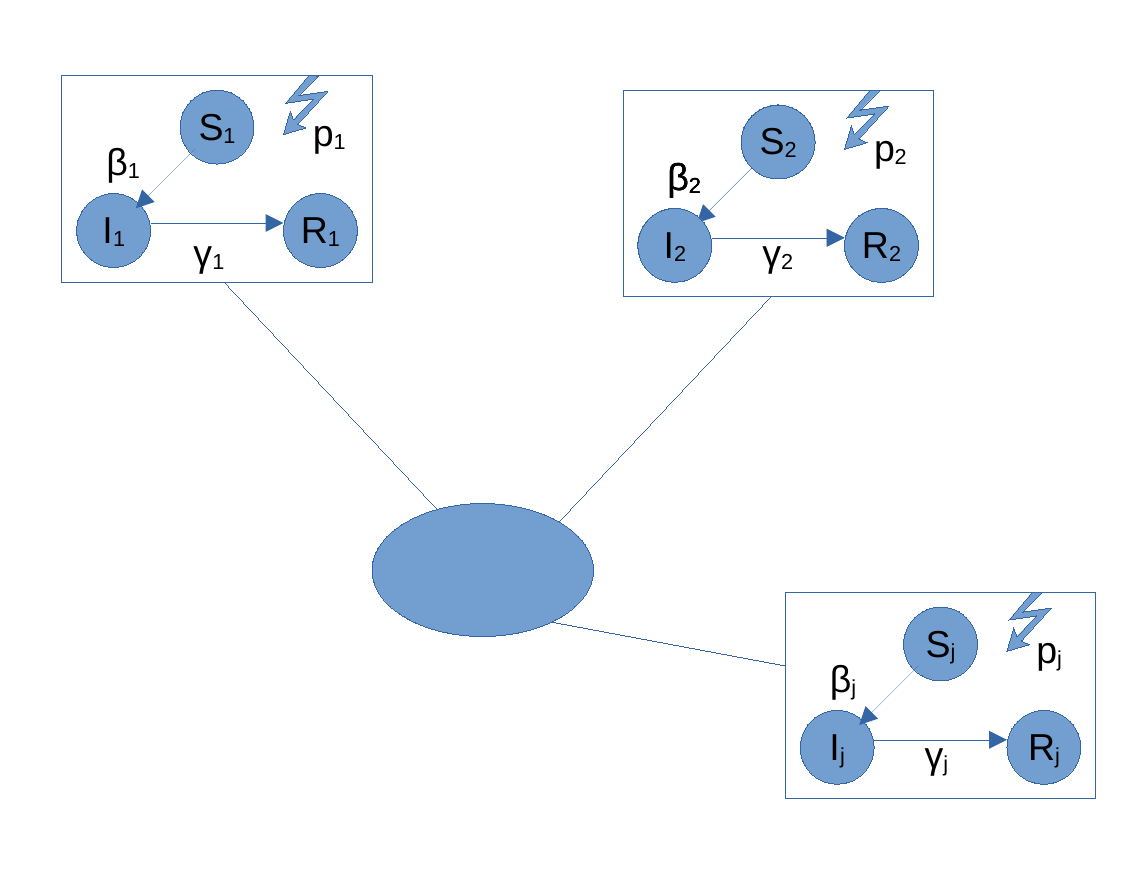}}
\caption{SIR Model With Protection Network.}
\label{Fig1}
\end{figure}
{The central node represents access to a common \emph{external} network, serving as a potential entry point for external infections. Star-like edges correspond to internal networks of individual companies, each governed by an independently controlled Susceptible-Infected-Recovered (SIR) mechanism for infection spread. These edges (indexed by $j$) incorporate continuous protection controls $p_j$, representing anti-malware or firewall measures that influence external infection tolerance and premium costs. Infections manifest at jump times of an independent Poisson process, modeling cyber-attack claims or ransom demands. The insurer's capital grows through state- and control-dependent premiums, while claims aggregate random per-infection costs multiplied by new infection counts. The insurer seeks to maximize a utility function (or reputation index) through an $\inf\sup$ formulation over possible insured companies' decisions, focusing on the supremum in time rather than expected values.

\textbf{Mathematical Context.} This problem extends Hamilton-Jacobi (HJ) theory on networks, building on deterministic dynamics frameworks in \cite{AchdouCamilliCutri2013}, \cite{AchdouOudetTchou2013}, and \cite{BarlesBrianiChasseigne2013}, but leaves junction discontinuities unexplored, for this particular dynamics. \\The $\mathbb{L}^\infty$ cost optimization originates from \cite{barron_ishii_89}, approximating such costs as limits of $\mathbb{L}^q$ norms as $q\rightarrow\infty$. Viscosity solutions for penalized problems are discussed in \cite{crandall_ishii_lions_92}, with state-dependent control constraints requiring extended viscosity theory \cite{barles_94}. For stochastic extensions, \cite{BDR_1994} introduces obstacle methods for dynamic programming principles, while \cite{BPZ_2015} and \cite{KPZ_2018} explore finite- and infinite-horizon adaptations using auxiliary variables.}

In this paper, we deal with the control of jump-type dynamics, generically presented as \begin{equation}
\label{SDEgen0}
d\bar{X}(t)=f\pr{\bar{X}(t),\bar{u}(t)}dt+\int_{\mathbb{R}}g\pr{\bar{X}(t-),\bar{u}(t),y}N(dt,dy),\ t\geq 0.
\end{equation}
Here, $f$ is a piecewise-deterministic controlled velocity, while the jumps are given by a compound Poisson mechanism, but depend on $\bar{X}$. 
{In the insurance model, $f$ represents epidemic spreading within individual companies, coupled with premium calculations for the insurer. External infections, transmitted through the central node, correspond to claims triggered by jump events. Unlike traditional models relying on expected values, this framework employs an (essential) supremum over both time and the random variable $\omega$. This approach prioritizes best-case scenarios rather than average outcomes, as explained in Section \ref{Sec2.4}.\\
The exact formulation of our control problem is in a relaxed sense 
\begin{equation}\label{V0}
V(\bar{x}_0):=\inf_{\gamma\in \Theta_1(1,\bar{x}_0)}\norm{(a,\bar{x})\mapsto aL\pr{\bar{x}}}_{\mathbb{L}^\infty\pr{d\gamma}},
\end{equation}where 
$\bar{x}_0$ corresponds to the initial configuration, and $\bar{x}$ to the controlled trajectory, the adjoined $a$ variable stands for a discount parameter $e^{-ht}$ necessary to deal with the infinite horizon, the measures $\gamma\in \Theta_1(1,\bar{x}_0)$ correspond to relaxed controls (see Section \ref{Subsec3.3} for the precise definitions). Roughly said, $V$ corresponds to the more standard formulation \[\inf_{\bar{u}\in\bar{\mathcal{U}}}\norm{\pr{\omega,t}\mapsto e^{-ht}L\pr{\bar{X}^{\bar{x}_0,\bar{u}}(\omega,t)}}_{\mathbb{L}^\infty\pr{\Omega\times\mathbb{R}_+,\mathbb{P}\otimes e^{-t}dt;\mathbb{R}}}.\]}

Before specifying the research program developed in the paper, we need to provide some details on the state-of-the art concerning the type of relaxation we have just described. {The relaxation technique aims at embedding the trajectories in what is called \emph{occupation measures} quantifying the time a system spends in specific states under given controls, enabling the transformation of dynamic control problems into convex optimization frameworks via moment-based constraints. These measures bridge continuous-time dynamics with static optimization by encoding state trajectories and control policies into probability distributions, facilitating analysis through linear programming and semidefinite relaxations.} Occupation measures are particularly adapted to non-convex dynamics in connection with problems presenting an asymptotic behavior. 
For standard control problems, when the costs are continuous, either expressing the cost with respect to controls or employing measure sets leads to the same value. The reader is referred to \cite{gaitsgory_92}, \cite{gaitsgory_04}, \cite{gaitsgory_quincampoix_09} in the deterministic setting,  \cite{GHS_2021} for reflected dynamics,
to \cite{borkar_gaitsgory_07}, \cite{G2} for diffusions, \cite{goreac_serea_ZubovPDMP}, or \cite{RS_2020} for various jump-type dynamics.  When used carefully, such sets of measures can be endowed with the compactness lacking on trajectories in the stochastic setting. Furthermore, the dual formulations of such problems are naturally associated with the Hamilton-Jacobi-Bellman equations satisfied by the value functions, see, for instance, \cite[Lemma 3.2]{gaitsgory_quincampoix_09}.

\subsection{Main Research Objectives and Structure of the Paper}

\textbf{Theoretical Objectives}\\

From a theoretical point of view, we wish to characterize the relaxed value function \eqref{V0}:
\begin{enumerate}
\item through a dual formulation as the point-wise supremum of eventually regular subsolutions of associated Hamilton-Jacobi equations (see Theorem \ref{ThMain});
\item as a viscosity solution of the natural Hamilton-Jacobi integro-differential inequality
\begin{equation}\label{HJB0}
\max\set{-h\psi(x)+H\pr{\psi,x,\psi(x),\nabla\psi(x)},L(x)-V(x)}=0,
\end{equation}
for a discontinuous Hamiltonian \begin{equation*}\begin{split}
H(\psi,x,r,p)&=\inf_{u\in U(\psi,x,r)}\scal{p,f(x,u)},\\&\textnormal{ where }U(\psi,x,r)=\set{u\in U:\ \psi(x+g(x,u,y))\leq r,\ \mathbb{P}_{C_1}-a.s.},\end{split}
\end{equation*}
for $x\in\mathbb{R}^m,\ r\in\mathbb{R},\ p\in\mathbb{R}^m$ and $\psi$ a bounded uniformly continuous function ($\psi  \in BUC(\mathbb{R}^m;\mathbb{R}$). The precise definition of viscosity sub/super-solution will involve the upper/lower limits of $H$ (see Definition \ref{DefSol}). The law $\mathbb{P}_{C_1}$ models claims and corresponds to the non-temporal component in the compensator of the compound Poisson process.
\end{enumerate}
Both characterizations rely on the $\mathbb{L}^q$-cost value functions $V_q$ approximating $V$. The latter characterization requires {equi-continuity of the family of value functions  $(V_q)_{q>1}$} guaranteed under reasonable assumptions on the post-jump positions $x+g(x,u,y)$, {cf. Assumption \ref{Ass_g_1}.}\\

{The main contributions of the paper are the following.
\begin{enumerate}
    \item We establish a novel dual characterization for the approximating value functions $V_q$. Theorem \ref{PropTool1} Assertion 3 provides entirely new results, explicitly crafted for the $q\rightarrow\infty$ limit transition, which lie outside the scope of \cite{goreac_serea_ZubovPDMP}.
    \item The convergence of $V_q$ to $V$ and its dual characterization give the first key result in Theorem \ref{ThMain}.
    \item   
Under Assumption \ref{Ass_g_1}, we establish a connection between $ V $ and \eqref{HJB0} in Section \ref{SecClassSol}, adapting methods from \cite{barron_ishii_89} to non-local Hamiltonians. Equi-continuity of the family  $\{V_q\}_q $ is demonstrated via Remark \ref{RemgPartCase}, a critical step for deriving the limit Hamiltonian (Proposition \ref{PropU}, Assertion 2; Theorem \ref{ThClassSol}). 
\end{enumerate}
{
During the review process, it was noted that providing additional details on the frequently cited paper \cite{goreac_serea_ZubovPDMP} could enhance the reader's understanding.
\begin{enumerate}\item The work in \cite{goreac_serea_ZubovPDMP} addresses more general dynamics, and its arguments can be directly applied to functions of the form $V_q^q$, broadly speaking. For this reason, we have relegated the proof of Theorem \ref{ThGaclco} to the Appendix, as it primarily reproduces results from \cite{goreac_serea_ZubovPDMP}. The main limitations are:
\begin{enumerate}
\item The modulus of continuity for such functions (see \cite[Proposition 6]{goreac_serea_ZubovPDMP}, which relies on \cite[Theorem 3.6]{Soner86_2}), and consequently for $V_q$, is not given explicitly;
\item The duality characterization is established for $V_q^q$, rather than directly for $V_q$.\end{enumerate}
\noindent  In contrast, the present paper:
    \begin{enumerate}
        \item Provides detailed trajectory estimates for shaken systems and a direct dual characterization of $V_q$ in Subsection \ref{SubsecStep1}. Our analysis explicitly quantifies the dependence on $q$ in the $L^q$ cost functions, which is crucial for the limit transition as $q \to \infty$. This is particularly highlighted in Remark \ref{RemgPartCase}, and the uniform continuity plays a key role in Section \ref{SecClassSol}.
        \item Establishes a dual characterization of $V_q$ (Theorem \ref{PropTool1}, Assertion 3) that is independent from the duality result for $V_q^q$ in \cite{goreac_serea_ZubovPDMP}, although some proof techniques share similarities. Furthermore, its formulation involves the candidates $\psi$ evaluated at $\bar{x}_0$ (please compare with Theorem \ref{ThGaclco}, assertion 2, and a restriction more likely to converge as $q\rightarrow\infty$.\footnote{Please take a look at some more precise details after Theorem \ref{ThMain}.}
    \end{enumerate}
\item The compactness of occupation measures in the 2-Wasserstein space is derived from the trajectory estimates in Proposition \ref{PropEstimTraj} combined with the bounds in \eqref{Theta}. Unlike \cite[Corollary 8]{goreac_serea_ZubovPDMP}, which only guarantees compactness with respect to weak convergence under constraints, our result ensures compactness in the Wasserstein metric. The moment estimates are essential in order to efficiently implement primal/dual algorithms and numerically find the optimal strategies, see \cite{GLWX_AMC_2024}.
\end{enumerate}
}

\noindent \textbf{Model Objectives}\\

{The second objective of our work is to provide practical motivation for the theoretical inquiries outlined earlier. To this end, we introduce in Section \ref{Sec2Model} a mathematical framework for cyber-risk modeling intrinsically linked to the network topology shown in Figure \ref{Fig1}. The model features an insurer covering claims for interconnected companies (represented as edges) sharing a common external network (central node). Internal infection spread follows a piecewise deterministic SIR dynamic, where the infected count in each edge evolves through jump mechanisms driven by the network’s average infectiousness and modulated by a continuous \emph{protection level} control. Implementing protection generates premium income for the insurer, while infection-driven jumps trigger claims that directly affect capital reserves. The insurer’s \emph{reputation level} is quantified as the running supremum of an utility function of its capital, minimized over all admissible control policies. Insured entities retain full agency over risk management outcomes through two key actions: (i) strategic selection of protection protocols, (ii) implementation of internal containment measures to mitigate infection propagation, and, implicitly through evaluative feedback directed at the insurer—collectively shaping operational and contractual dynamics.}\\

\noindent \textbf{Structure of the paper}\\

The paper is organized as follows. In the remaining of the section we introduce some notations. Section \ref{Sec2Model} provides the details on the insurance of cyber-risks model we have in mind. Section \ref{Sec3} introduces the rigorous relaxed formulation of the problem for generic controlled dynamics \eqref{SDEgen0}.  In particular, we explain the role of the discount as an adjoined variable and coherence with more general piecewise deterministic Markov processes as considered in \cite{goreac_serea_ZubovPDMP}.  Subsection \ref{Subsec3.2} provides a description of the "shake-of-coefficients" approach and the estimates on trajectories. These allow a proper formulation of the set of constraints $\Theta(a_0,x_0)$ and the consistency of its marginals in $\Theta_1(a_0,x_0)$ in \eqref{Theta} and Theorem \ref{ThGaclco}. The first main result is stated in Section \ref{SecMain1}, Theorem \ref{PropTool1} and the limiting result in Proposition \ref{PropV=supVq} making use of the compactness of the measure set of constraints. Finally, the adaptation of the program in \cite{barron_ishii_89} to the non-local framework makes the object of Section \ref{SecClassSol}. We identify the limit Hamiltonian in Proposition \ref{PropBI}. The state-depending sets of controls and the relaxed Hamiltonians make the object of Proposition \ref{PropU}, and the viscosity result is given in Theorem \ref{ThClassSol}. {Section \ref{SecNew} provides further details on how the standing assumptions integrate with the proposed cyber risk model.} Finally, for our readers' comfort, the intermediate, or more standard proofs are gathered in the Appendix.

\subsection{Notations}
Throughout the paper we make use of several notations resumed hereafter.
\begin{itemize}
\item We let $\mathbb{R}_+$ stand for the non-negative real numbers and $\mathbb{T}:=\set{\pr{s,i}\in\mathbb{R}_+^2:\ s+i\leq 1}$ denote a fundamental triangle;
\item Given an integer $n\geq 1$, we let $\mathbb{T}_n:=\set{\pr{s_1,s_2,\ldots,s_n,i_1,i_2,\ldots,i_n}:\ \pr{s_j,i_j}\in\mathbb{T},\ \forall 1\leq j\leq n}$;
\item The set of controls $U:=\pp{u_{\min},u_{\max}}$ and $\mathcal{P}rev\subset [1,\infty)^n$ are assumed to be compact.
\item For the theoretical results,  we use the following.
\begin{itemize}
\item The state space is considered to be some subset of the Euclidean space $\mathbb{R}^m$, where $m$ is a positive integer.
\item The set $B_\delta(x)$ denotes the $\delta>0$-closed ball around $x\in\mathbb{R}^m$.
\item The control space $\bar{U}$ is a (subset of a) compact metric space. \footnote{{It is worth mentioning that, in the example of cyber risks discussed in Section \ref{Sec2Model}, the generic control set $\bar{U}$ represents the product set $\bar{U} = U \times \mathcal{P}rev$. Furthermore, we denote by $\bar{\mathcal{U}}$ the family of predictable, $\bar{U}$-valued processes that appear in the generic dynamics described in equation \eqref{SDEgen}.}}
\item The family $C_b^1$ stands for the family of real-valued, continuously differentiable functions that are, together with their first-order derivatives, bounded. If, instead of boundedness, on asks for quadratic growth, the resulting set is $C_2^1$.
\item The set $\mathcal{P}(\cdot)$ stands for probability measures on some metric space given as argument and endowed with its Borel $\sigma$-field.
\item For a probability measure $\gamma\in\mathcal{P}(\cdot)$, we lett $Supp\ \gamma$ denote its support.
\item For a random variable $C_1$ modeling claim sizes, we let $\mathbb{P}_{C_1}$ denote its law.
\end{itemize}
\end{itemize}
\section{A Controlled Jump Model for Cyber-Risks With Firewalled Edges and SIR Intra-Edge Spreading}\label{Sec2Model}

\subsection{Heuristic Cyber-Risk Model}

In our proposed model, the network comprises $n \geq 1$ entities, each with a domestic network (edges of an interconnecting network, indexed by $j \in \{1,\ldots,n\}$). Contagion within each edge follows an SIR model with entity-specific contact ($\beta_j$) and recovery ($\gamma_j$) rates.\\
\noindent\textbf{Key Deterministic Dynamics} The SIR (Susceptible-Infected-Recovered) model is a foundational epidemiological framework describing disease spread through three compartments: (a) susceptible ($S_j$): individuals at risk of infection, (b) infected ($I_j$): infectious individuals transmitting the disease, and (c) recovered ($R_j$): individuals who have recovered (or died) and gained immunity. \\ 
The model uses two entity-specific parameters: \emph{contact rate} ($\beta_j$): reflects transmission intensity, combining interaction frequency and pathogen transmissibility, and \emph{recovery rate} ($\gamma_j$): defines the inverse of the average infectious period ($\gamma_j = 1/D_j$, where $D_j$ is duration). \\
Disregarding the jumps, the deterministic evolution within each edge follows:  
$
\frac{dS_j}{dt} = -\beta_j S_j I_j, \quad \frac{dI_j}{dt} = \beta_j S_j I_j - \gamma_j I_j, \quad \frac{dR_j}{dt} = \gamma_j I_j.
$\\ 
\noindent\textbf{Control Implications}
Adjusting $\beta_j$ (via reduced contacts or transmissibility) or $\gamma_j$ (via treatments) alters the \emph{basic reproduction number}, enabling containment strategies like prophylactic or curative measures. In our model, we only control, in a multiplicative manner, $\beta_j$ changing to $\beta_ju(t)$.\\
\noindent\textbf{Key Jump Dynamics}  
There are three main features: \textit{intra-infectiousness} governed by the infected proportion $i_j$ in each edge, \textit{inter-infectiousness} computed as the average $i^0 := \frac{\sum_{j=1}^n i_j}{n}$, and \textit{controls} since each entity selects a continuous protection level $p_j$ (firewall) and intra-exposure reduction $u_j$.  \\
\noindent\textbf{Infection Mechanism}  
Extra infections occur when $i_j p_j \leq i^0$, triggering a jump in intra-infectiousness to $i_j \leftarrow \frac{i^0}{p_j}$. Protection incurs a premium cost and enables claims proportional to changes in $i_j$, influencing an insurer’s reserve $x$. The insurer’s \textit{reputation index} is defined as the historical maximum of $x$, minimized over independent control parameters $(p_j, u_j)$. 

\subsection{Description of the Simplified Model}
The network presents a single node with $n\geq 1$ edges. Each edge $1\leq j\leq n$ has:  
  \begin{itemize}  
    \item SIR dynamics with edge-specific parameters,  
    \item Protection level $p_j$ (firewall), e.g., medium ($p_j=1$) or high ($p_j=\alpha>1$); the admissible protections are described by $\mathcal{P}rev$.  
  \end{itemize}  
We further assume constant data volume per component ($r_j=1-s_j-i_j$), yielding a $2n$-dimensional state space $(s_1,\ldots,s_n,i_1,\ldots,i_n)$, with $n$-dimensional control $(s,i,p)\in\mathbb{T}_n\times\mathcal{P}rev$.  The insurer premiums account for: \emph{internal exposure} quantified through the infected population,  the \emph{external protection} given by the firewall level $p_j$, and the \emph{lock-down}, i.e., the legislated or collective policy reducing infections (e.g., COVID-19 premium reimbursements\footnote{French insurers MAIF, GMF, Matmut}). \\  
The global \emph{premium} is computed via a functional  
$
c:\mathbb{T}_n \times [u_{\min},1] \times \mathcal{P}rev \rightarrow \mathbb{R}_+,  
$
where $\mathbb{T}_n$ represents the system state space, $[u_{\min},1]$ defines admissible internal exposure controls, and $\mathcal{P}rev$ encodes protection levels.  
We consider \emph{Poisson-driven updates} through which the edge status updates occur at jump times of a Poisson process (rate $\lambda>0$), e.g., ransom demands.  \\
Concerning the \emph{infection mechanism}: 
   \begin{itemize}
       \item Edge $j$ is infected if its \emph{internal infectiousness} $i_j$ falls below the node’s \emph{global average infectiousness} $i^0 := {\frac{\sum_{j=1}^n i_j}{n}}$ modulated by protection $p_j$, i.e., 
    $ i_j<\frac{i^0}{p_j}$. For instance, under medium protection ($p_j=1$), infection occurs when $i_j < i^0$.
    \item Upon infection, the following occur. The internal infectiousness updates to $i_j \leftarrow \frac{i^0}{p_j}$ (scaled by protection). Susceptible proportion updates to $s_j \leftarrow \min\set{s_j, 1 - \frac{i^0}{p_j}}$ (ensuring population normalization $s + i + r \leq 1$). Re-infections (recovered → infected) are modeled as instantaneous transitions.\\
The re-infections are assumed instantaneous to avoid tracking recovered individuals ($r = 1 - s - i$). Let us note that higher $p_j$ (e.g., $p_j = \alpha > 1$) reduces infection likelihood by impacting the threshold $\frac{i^0}{p_j}$.
\end{itemize}
\subsection{Adding a Capital Component; the Reference Mathematical Model}
Let $(S,I):=\pr{S_1,S_2,\ldots, S_n,I_1,I_2,\ldots, I_n}$ be the stochastic system describing the evolution of the states of the system.  Furthermore, let $N$ be a Poisson process with intensity $\lambda>0$. Then, given a predictable lock-down policy $u$ and a predictable decision on the external protection levels $p\in \mathcal{P}$, one has 
{
\begin{equation}\label{SIR_dyn}\begin{cases}
dS_j^{u,p}(t)=&-\beta_ju(t)S_j^{u,p}(t)I_j^{u,p}(t)dt-\pr{S_j^{u,p}(t-)-1+\frac{I^{0,u,p}(t-)}{p_j(t)}}^+dN(t);\\
dI_j^{u,p}(t)=&\pr{\beta_j u(t)S_j^{u,p}(t)-\gamma_j}I_j^{u,p}(t)dt+\pr{\frac{I^{0,u,p}(t-)}{p_j(t)}-I_j^{u,p}(t-)}^+dN(t);\\
dR^{u,p}_j(t)=&\gamma_jI_j^{u,p}(t)dt\\&+\pp{\pr{S_j^{u,p}(t-)-1+\frac{I^{0,u,p}(t-)}{p_j(t)}}^+-\pr{\frac{I^{0,u,p}(t-)}{p_j(t)}-I_j^{u,p}(t-)}^+}dN(t),\\
I^{0,u,p}(t)=&\frac{1}{n}\sum_{1\leq j\leq n}{I_j^{u,p}(t)}.
\end{cases}
\end{equation}}Here, $x^+=\max\set{x,0}$. This is completed with the equation of the capital of the insurance company.  When $u=1$, there is no access restriction policy available and the claims are given by mutually independent random variables (independent as well of the Poisson process) denoted by $C_k$. Their law will be designated by $\mathbb{P}_C$ ($C$ being the generic random variable) and supported on $\mathbb{R}_+$.  The associated {compound} Poisson process is synthesized by the random measure $N(dt,dy)$ on $\mathbb{R}_+^2$. Then, if the claim were paid per \Db{new} infection, this would lead to a total payment on $\pp{0,t}$ of type \[\sum_{1\leq j\leq n}\int_0^t\int_{\mathbb{R}_+}\pr{\frac{\sum_{1\leq \Db{k}\leq n}I_{\Db{k}}(l-)}{np_j(l)}-I_j(l-)}^+yN(dl,dy).\]
When the (internal) contact rate is reduced, due to limitations of access to the individual networks, this is likely to reduce the claims. To simplify things, we consider a proportional reinsurance-like model for the claims, i.e., given the predictable strategy $u$, the claims are of type \[\sum_{1\leq j\leq n}\int_0^t\int_{\mathbb{R}_+}\pr{\frac{\sum_{1\leq \Db{k}\leq n}I_{\Db{k}}(l-)}{np_j(l)}-I_j(l-)}^+u(l)yN(dl,dy).\]Putting all these together, one gets the following equation for the capital
\begin{equation}
\label{reserve}
dX^{u,p}(t)=c(S(t),I(t),u(t),p(t))dt-\sum_{1\leq j\leq n}\int_{\mathbb{R}_+}\pr{\frac{\sum_{1\leq \Db{k}\leq n}I_{\Db{k}}(t-)}{np_j(t)}-I_j(t-)}^+u(t)yN(dt,dy).
\end{equation}
This is, of course, a multiplicative noise model in which claim $C_k$ is distributed to all the newly infected individuals, regardless of their edge.  In this paper we focus on improving a \emph{reputation index} for the insurance company uniformly in the policies $(u,p)$ on which the insurance company has no saying. Roughly speaking, we do not care about the capital position becoming negative (this can be reset to $0$ by capital injection), but this may reflect negatively on the reputation.

\subsection{Maximizing the Insurer's Reputation. A Relaxed Formulation}\label{Sec2.4}
We consider a classical non-negative, Lagrangian $L(s,i,x)$ acting as a reputation/performance index. \Db{A reputation index for an insurance company aggregates data from various sources, such as customer reviews, financial performance, and industry rankings, to measure how the company is perceived by stakeholders. Online reviews play a key role in this process, reflecting customer satisfaction and influencing trust and visibility. Rankings derived from such indices can have lasting value; for example, a company recognized as "Best in Claims Handling in 2020" can continue to leverage this accolade in marketing materials to bolster its credibility and attract new clients even years later. This long-term utility underscores the importance of maintaining a strong reputation and consistently high performance.\\
The supremum cost framework addresses optimization problems where the maximum cost is minimized across all possible decisions taken by the policyholders. When focusing on review-driven metrics, policyholders frequently attribute systemic dissatisfaction to insurers regardless of individual claim circumstances. This tendency manifests in reputation indices through minimization components that penalize negative feedback patterns.\\
The ultimate control lies with policyholders, who shape market dynamics through their preferences (in this particular framework, curative measures and external protection). Insurers strategically adopt the most favorable index results to enhance competitiveness and attract clientele.}
In connection to this, we consider the problem of maximizing, over time, (and $\omega\in \Omega$)  \[L\pr{S^{S_0,I_0,u,p}(t),I^{S_0,I_0,u,p}(t),X^{S_0,I_0,X_0,u,p}(t)}.\]

\noindent We will consider a relaxed formulation for several reasons.\\
\noindent The most relevant stochastic case (see Proposition \ref{PropPremium}) is when the dependence of $u$ is affine, i.e.  $c(s,i,u,p)=c(s,i,p)u$ and $\mathcal{P}rev=\set{1,\alpha}$ consists of two scenarios. In this case, the post-jump is not convex in $p$\footnote{The reader is invited to take a look at the jump occurring in $I$; this is also valid if $\mathcal{P}rev$ is an interval.} and the existence of optimal controls cannot be obtained using standard arguments.\\
\noindent On the other hand, if the number of edges $n=1$, then the system is reduced to a deterministic one.  In this case,  by adapting the procedure in \cite{barron_ishii_89} from the uniform law on a compact time interval to the exponential $e^{-t}dt$ one on $\mathbb{R}_+$,  the $\mathbb{L}^\infty\pr{\mathbb{R},e^{-t}dt;\mathbb{R}}$-norm can be approximated with $\mathbb{L}^q\pr{\mathbb{R},e^{-t}dt;\mathbb{R}}$-norms as $q\rightarrow\infty$. \\

\noindent Unlike previous works on the subject in stochastic frameworks (\cite{BDR_1994}, \cite{KPZ_2018}), our stochasticity is of jump nature and we choose not to seek to minimize over predictable $(u,p)$ the expectation of the running maximum $\mathbb{E}\pp{\underset{t}{esssup}\ L}$ but rather $\underset{t,\omega}{esssup}\ L$. As such, we do not use a further variable to infer the dynamic programming principle, but rather the structure properties of the \emph{occupation measures} as a natural tool for asymptotic-involving control problems.

\section{Relaxed Formulations for Control Problems With Jump Dynamics}\label{Sec3}
\subsection{Generic Dynamics}
Prior to giving our theoretical considerations and in order both to simplify notations and to work under more generality, let us consider a generic control system of the form 
\begin{equation}
\label{SDEgen}
d\bar{X}(t)=f\pr{\bar{X}(t),\bar{u}(t)}dt+\int_{\mathbb{R}}g\pr{\bar{X}(t-),\bar{u}(t),y}N(dt,dy),\ t\geq 0.
\end{equation}
Here, $N$ corresponds to a {compound} Poisson process and the compensator is $\hat{N}(dtdy)=\lambda\mathbb{P}_{C_1}(dy)$, where the random variable $C_1$ models claim sizes,  and $\mathbb{P}_{C_1}$ denotes its law. The filtration $\mathbb{F}$ will be the natural one generated by $N$ and completed with the $\mathbb
{P}$-null sets, and notions such as predictability refer to this filtration. \\
Admissible controls form a class $\bar{\mathcal{U}}$ and they refer to predictable $\bar{U}$-valued processes.\\
We will enforce the following assumptions.
\begin{ass}\label{Ass_gen}
\begin{enumerate}
\item The state space is $\mathbb{R}^m$, for some positive integer $m\geq 1$, and the control space $\bar{U}$ is a compact (subspace of a) metric space.
\item The drift coefficient $f:\mathbb{R}^m\times \bar{U}\longrightarrow\mathbb{R}^m$ is uniformly continuous, bounded by $\norm{f}_0$ and $\pp{f}_1$-Lipschitz continuous in the state variable $\bar{x}$, uniformly with respect to $\bar{u}\in\bar{U}$.
\item The jump coefficient $\Db{g}:\mathbb{R}^m\times \bar{U}\times \mathbb{R}\longrightarrow\mathbb{R}^m$ is uniformly continuous. For every $y\in \mathbb{R}$, $g\pr{\cdot,\cdot,y}$ is bounded by $\norm{g}_0(y)$ and $\pp{g}_1(y)$-Lipschitz continuous in the state variable $\bar{x}$, uniformly with respect to $\bar{u}\in\bar{U}$.
\item The functions $\norm{g}_0(\cdot)$ and $\pp{g}_1(\cdot)$ are fourth-order-integrable w.r.t. $\mathbb{P}_{C_1}$.\footnote{Actually, any order exceeding $2$ suffices to guarantee that the associated occupation measures belong to the Wasserstein space $\mathcal{P}_2$.}
\end{enumerate}
\end{ass}
Under this assumption, for every $\bar{x}_0\in\mathbb{R}^m$ and every admissible control $\bar{u}\in \bar{\mathcal{U}}$, the equation \eqref{SDEgen} admits a unique solution denoted by $\bar{X}^{\bar{x}_0,\bar{u}}$.\\
\noindent Let us fix, for the time being, $h>0$ large enough. We will extend the dynamics by adding the equation
 \begin{equation}
 da^{a_0}(t)=-ha^{a_0}(t)dt,\ a(0)=a_0\in\mathbb{R}.
 \end{equation}
 \begin{remark}
 \begin{enumerate}
 \item With the notations of \cite{goreac_serea_ZubovPDMP},  one can consider this system $\pr{a,\bar{X}}$ as a piecewise deterministic Markov process in which \begin{itemize}
 \item one adds a discrete component $\gamma\in\set{0,1}$ systematically switching in order to indicate a non-degenerate post-jump transition measure; This renders the system $\pr{\gamma,a,\bar{X}}$ a true (controlled) piecewise deterministic Markov process in the sense of \cite{davis_93};
 \item the deterministic velocity leaves $\gamma$ unchanged and is given by $\bar{f}(\gamma,a,\bar{x},\bar{u})=\pr{0,-ha,f\pr{\bar{x},\bar{u}}}$;
 \item the function $\bar{\lambda}=\lambda$ is constant;
 \item the jump is systematically switching in $\gamma$, leaves $a$ unchanged and is completed with the $g$ from before, i.e., \[\bar{Q}(\gamma, a,\bar{x},\bar{u},d\gamma',da',d\bar{y})=\delta_{1-\gamma}(d\gamma')\delta_{a}(da')\int_{\mathbb{R}}\delta_{\bar{x}+g\pr{\bar{x},\bar{u},y}}\pr{d\bar{y}}\mathbb{P}_{C_1}(dy).\]
 \end{itemize}
 Here and elsewhere, $\delta$ denotes a Dirac mass, while $d\gamma'\times da'\times d\bar{y}\subset\set{0,1}\times \mathbb{R}\times\mathbb{R}^{m}$ is a Borel set.
\item  $\lambda Q$ corresponds to the usual jump part, i.e., given $\phi:\mathbb{R}^{m}\longrightarrow\mathbb{R}$ bounded and uniformly continuous, one has  
\begin{align*}\int_{\mathbb{R}^{m}}\pr{\phi(\bar{y})-\phi(\bar{x})}(\lambda Q)(\gamma, a,\bar{x},\bar{u},\set{0,1},\mathbb{R},d\bar{y})=\lambda\pp{\int_{\mathbb{R}}\phi\pr{\bar{x}+g\pr{\bar{x},\bar{u},y}}\mathbb{P}_{C_1}(dy)-\phi(\bar{x})}.
\end{align*}
\item The assumptions on \cite[Page 213]{goreac_serea_ZubovPDMP} are trivially satisfied: (A1) follows from Assumption \ref{Ass1}; (A2) is trivial as $\lambda$ is constant; (A3) is linked to $g$ being Lipschitz-continuous in the state variable uniformly in control and $y$; (A4b) follows from $g$ being bounded, while (A4a) is not needed in this paper (as it primarily targets stability issues). As a consequence, we can refer to the results in \cite{goreac_serea_ZubovPDMP} on linearization techniques. As an alternative, one can take a look at \cite{Serrano_2015}.
\end{enumerate}
 \end{remark}
 \subsection{Some Basic Estimates}\label{Subsec3.2}
 Inspired by Krylov's shaking of coefficients, cf. \cite{Krylov_step_99}, \cite{krylov_00} (see also \cite{barles_jakobsen_02}) for diffusions, let us proceed in several steps to emphasize the properties we are going to need for approximating trajectories or value functions.\\
 
\noindent\textbf{Step 1.} Extending the dynamics

We consider $f^+(\bar{x},\bar{u},e):=f(\bar{x}+e,\bar{u})$ and $\Db{g^+}(\bar{x},\bar{u},e)=e+g(\bar{x}+e,\bar{u})$, where $e\in\mathbb{R}^m$; we will actually be interested in \[e\in B_1(0):=\set{x\in\mathbb{R}^m:\ \abs{x}\leq 1}. \]One gets $f^+$ that is $\pp{f}_1$-Lipschitz-continuous in space, uniformly with respect to $(\bar{u},e)\in \bar{U}\times \mathbb{R}^m$ and bounded by $\norm{f}_0$. The same kind of assertions hold true for $g^+$ (the same Lipschitz constant $\pp{g}_1$, and a bound of type $1+\norm{g}_0(y)$, which, in order not to complicate the notations, we will still denote by $\norm{g}_0(y)$ with the same integrability properties as the initial function).\\

We are now able to consider the control system
\begin{equation}\label{SDEext}\begin{cases}
d\bar{X}^{+}(t)=f^+\pr{\bar{X}^{+}(t),\bar{u}(t),e(t)}dt+\int_{\mathbb{R}}g^+\pr{\bar{X}^{+}(t-),\bar{u}(t),e(t),y}N(dt,dy),\\
\bar{X}^{+}(0)=0,
\end{cases}
\end{equation}
with $\bar{u}\in\bar{\mathcal{U}}$ and $e$ a predictable $B_1(0)$-valued control process. As before, the solution can be written with an explicit dependence on initial datum and controls as $X^{+,\bar{x}_0,\bar{u},e}$.\\

\noindent\textbf{Step 2.} We are going to exhibit some basic estimates for the behaviour of the solutions.
\begin{proposition}\label{PropEstimTraj}
Let us consider that Assumption \ref{Ass_gen} holds true. Then the following assertions are valid.
\begin{enumerate}
\item For every initial data $\bar{X}_k(0)\in \mathbb{R}^m$, with $k\in\set{1,2}$, and every admissible control couple $(\bar{u},e)$, if $\bar{X}_k:=\bar{X}^{+,\bar{X}_k(0),\bar{u},e}$ denotes the associated solutions in equation \eqref{SDEext}, then 
\begin{equation}
\label{EstimX1-X2}
\mathbb{E}\pp{\abs{\pr{\bar{X}_1-\bar{X}_2}(t)}^2}\leq e^{\set{2\pp{f}_1+\lambda\int_{\mathbb{R}}\pp{2\pp{g}_1(y)+\pp{g}_1^2(y)}\mathbb{P}_{C_1}(dy)}t}\abs{\pr{\bar{X}_1-\bar{X}_2}(0)}^2,
\end{equation}
for all $t\geq 0$.
\item For every initial datum $\bar{X}(0)\in\mathbb{R}^m$, and every admissible control $\pr{\bar{u},e}$ such that $\abs{e}\leq 1,\ \mathbb{P}\times dt$-a.s. one has, for all $t\geq 0$,
\begin{equation}
\label{EstimX+-X}\begin{split}
&\mathbb{E}\pp{\abs{\pr{\bar{X}^{+,\bar{X}(0),\bar{u},e}-\bar{X}^{\bar{X}(0),\bar{u}}}(t)}^2}\\&\leq  \pr{\pp{f}_1+2\lambda\int_{\mathbb{R}}\pr{1+\pp{g}_1\pr{y}}^2\mathbb{P}_{C_1}(dy)}e^{\pr{3\pp{f}_1+\lambda\int_{\mathbb{R}}\pp{1+4\pp{g}_1(y)+2\pp{g}_1^2(y)}\mathbb{P}_{C_1}(dy)}t}\int_0^t\mathbb{E}\pp{\abs{e(l)}^2}dl.
\end{split}
\end{equation}
\item If $\bar{X}_0$ is a fixed initial datum, then, for every admissible (predictable) $\bar{U}\times B_1(0)$-valued control $\pr{\bar{u},e}$, 
\begin{equation}
\label{CompactEstim}
\begin{split}
\mathbb{E}\pp{\abs{\bar{X}^{+,\bar{X}_0,\bar{u},e}(t)}^4}\leq e^{\frac{\lambda+1}{2}t}\pr{\abs{\bar{X}_0}^4+5^4\max\set{\lambda,1}\pr{\norm{f}_0^4
+\int_{\mathbb{R}}\norm{g}_0^4(y)\mathbb{P}_{C_1}(dy)}t}.
\end{split}
\end{equation}
\item Let us assume, in addition, that $\bar{x}\mapsto \bar{x}+g\pr{\bar{x},\bar{u},y}$ is $1$-Lipschitz, for every $\pr{\bar{u},y}\in \bar{U}\times Supp\pr{\mathbb{P}_{C_1}}$. Then, for every $q\geq 2$, every initial data $\bar{X}_k(0)\in\mathbb{R}^m$, with $k\in\set{1,2}$, and every admissible control couple $(\bar{u},e)$, 
\begin{equation}
\label{EstimX1-X2_q}
\mathbb{E}\pp{\abs{\pr{\bar{X}^{+,\bar{X}_1(0),\bar{u},e}-\bar{X}^{+,\bar{X}_2(0),\bar{u},e}}(t)}^q}\leq e^{q\pp{f}_1t}\abs{\pr{\bar{X}_1(0)-\bar{X}_2(0)}}^q,
\end{equation}
for all $t\geq 0$. Furthermore,
\begin{equation}
\label{EstimX+-xg}
\mathbb{E}\pp{\abs{\pr{\bar{X}^{+,\bar{X}_1(0),\bar{u},e}-\bar{X}^{\bar{X}_1(0),\bar{u}}}(t)}^q}\leq \pr{\pp{f}_1+\lambda q}e^{\pr{(2q-1)\pp{f}_1+\lambda q}t}\int_0^t\mathbb{E}\pp{\abs{e(l)}^q}dl,
\end{equation}for all $t\geq 0$.
\end{enumerate}
\end{proposition}
The proof is quite standard; for our readers' sake, the details are postponed to the Appendix.
 \subsection{General Running Cost Problems and Their Linear Formulations}\label{Subsec3.3}
To an initial configuration $a_0\in\mathbb{R}$ and $\bar{x}_0\in\mathbb{R}^m$ and an admissible control $\bar{u}_0\in\bar{U}$, we associate the \emph{(expectation of the) occupation measure} $\gamma^{a_0,\bar{x}_0,\bar{u}_0}\in\mathcal{P}\pr{\mathbb{R}\times\mathbb{R}^m\times \bar{U}},$
 \begin{equation}
 \gamma^{a_0,\bar{x}_0,\bar{u}_0}(da,d\bar{x},d\bar{u}):=\mathbb{E}\pp{\int_0^\infty e^{-t}\mathbf{1}_{\set{a^{a_0}(t)\in da,\ \bar{X}^{\bar{x}_0,\bar{u}}(t)\in d\bar{x},\ \bar{u}_0(t)\in d\bar{u}}}dt},
 \end{equation}for all Borel sets $da\subset\mathbb{R},\ d\bar{x}\subset\mathbb{R}^m,\ d\bar{u}\subset\bar{U}$.\\
We assume that $\lambda< 1$, otherwise one needs to change the discount parameter to a quantity slightly larger than $\max\set{1,\lambda}$, but without affecting the reasoning.\\
The family of such occupation measures is denoted by $\Gamma\pr{a_0,\bar{x}_0}$. It can be seen as a subset of the family\begin{equation}\label{Theta}
 \begin{split}
 &\Theta\pr{a_0,\bar{x}_0}:=\\
 \bigg\{ &\gamma\in \mathcal{P}\pr{\mathbb{R}^{m+1}\times\bar{U}}:\ \forall \phi\in C_b^1\pr{\mathbb{R}^{m+1}},\\
&\begin{split}\phi\pr{a_0,\bar{x}_0}+\int_{\mathbb{R}^{m+1}\times \bar{U}}\bigg[ &-\phi(a,\bar{x})+\scal{\begin{pmatrix}
 -ha\\ f\pr{\bar{x},\bar{u}}
 \end{pmatrix},\nabla \phi(a,\bar{x})}\\&+\lambda\int_{\mathbb{R}}\pr{\phi\pr{a,\bar{x}+g\pr{\bar{x},\bar{u},y}}-\phi(a,\bar{x})}\mathbb{P}_{C_1}(dy)\bigg]\gamma(da,d\bar{x},d\bar{u})=0; \end{split}\\
 &Supp\ {\gamma}\subset\pp{-\abs{a_0},\abs{a_0}}\times \mathbb{R}^m\times \bar{U};\\
& \int_{\mathbb{R}^{m+1}\times \bar{U}}\abs{\bar{x}}^4\gamma\pr{da,d\bar{x},d\bar{u}}\leq \Db{\pr{\frac{2}{1-\lambda}}^2}\pp{\abs{\bar{x}_0}^4+5^4\pr{\norm{f}_0^4+\int_{\mathbb{R}}\norm{g}_0^4(y)\mathbb{P}_{C_1}(y)}}
 \bigg\}.
 \end{split}
 \end{equation}
 \Db{The equality constraint is nothing else than It\^{o}'s formula; the moment condition directly follows from the inequality in Proposition 2, assertion 3.
 Indeed, \begin{align*}&\int_{\mathbb{R}^{m+1}\times\bar{U}}\abs{\bar{x}}^4\gamma^{a_0,\bar{x}_0,\bar{u}_0}(da,d\bar{x},d\bar{u})=\int_0^\infty e^{-t}\mathbb{E}\pp{\abs{\bar{X}^{\bar{x}_0,\bar{u}_0}(t)}^4}dt\\ \leq &\abs{\bar{x}_0}^4\int_0^\infty e^{\frac{\lambda-1}{2}}dt+5^4\pr{\norm{f}_0^4
+\int_{\mathbb{R}}\norm{g}_0^4(y)\mathbb{P}_{C_1}(dy)}\int_0^\infty e^{\frac{\lambda-1}{2}t}tdt\\
&=\frac{2}{1-\lambda}\abs{\bar{x}_0}^4+\pr{\frac{2}{1-\lambda}}^25^4\pr{\norm{f}_0^4
+\int_{\mathbb{R}}\norm{g}_0^4(y)\mathbb{P}_{C_1}(dy)}.\end{align*}}
 Let us further emphasize that the uniform bound on the fourth-order moments on the $x$ component, together with the boundedness of $a$ imply the compactness with respect to the $2$-Wasserstein metric.\\ 
 {\begin{remark} The compactness of $U$ (together with the boundedness of $a$) plays a role in the sense that we do not impose a condition on the fourth moment in $u$. The case of unbounded controls can be treated similarly to the unboundness on the trajectory component $x$. \\
Concerning the compactness in $\mathcal{P}_2$,  this can be inferred from \cite[Theorem 5.5, Page 358]{CD1_2018}.
The only element needed is the uniform integrability in $\mathbb{L}^2$ (see (5.16) in the cited reference) which is valid, for instance, it the $\mathbb{L}^{2+\delta}$-moment is uniformly bounded for the family of measures. Ultimately compactness comes from Prohorov's theorem and tightness of the family of measures (see also \cite[Corollary 8]{goreac_serea_ZubovPDMP}), and the moment bound implies this tightness (when u and a live on compacts).
 \end{remark}}
 We remind the following useful characterization of $\Theta$ which synthesizes and details certain results in \cite{goreac_serea_ZubovPDMP}.
 \begin{theorem}
 \label{ThGaclco}
 We suppose Assumption \ref{Ass_gen} to hold true and $h>0$. Let $\pr{a_0,\bar{x}_0}\in \mathbb{R}_+\times \mathbb{R}^m$ be fixed. 
 \begin{enumerate}
 \item If $\Theta_1\pr{a_0,\bar{x}_0}$ denotes the space marginals of measures in $\Theta\pr{a_0,\bar{x}_0}$, i.e. \[\Theta_1\pr{a_0,\bar{x}_0}:=\set{\gamma\pr{\cdot,\bar{U}}:\ \gamma\in\Theta\pr{a_0,\bar{x}_0}},\]
 then \begin{equation}
 \label{EqThGaclco}
 \Theta_1\pr{a_0,\bar{x}_0}=\bar{co}\set{\gamma\pr{\cdot,\bar{U}}:\ \gamma\in\Gamma\pr{a_0,\bar{x}_0}}.
 \end{equation}
 Here, $\bar{co}$ denotes the Kuratowski closure with respect to the usual (weak) convergence of probability measures applied to the convex hull.  This closure can also be taken in the Wasserstein space $\mathcal{P}_2\pr{\mathbb{R}^{m+1}}$ with the classical $\mathcal{W}_2$ distance.
 \item If $\phi$ is a real-valued bounded uniformly continuous function on $\mathbb{R}^{m}$, then the following equalities hold true.
\begin{equation}
 \label{Eqinfsup}
 \begin{split}
& a_0^q\inf_{\bar{u}\in\bar{\mathcal{U}}}\mathbb{E}\pp{\int_{0}^\infty e^{-(1+qh)t}\phi\pr{\bar{X}^{\bar{x}_0,\bar{u}}(t)}dt}=\inf_{\gamma\in\Theta_1\pr{a_0,\bar{x}_0}} \int_{\mathbb{R}^{m+1}}a^q\phi(\bar{x})\gamma(da,d\bar{x})\\
&\begin{split}=\frac{a_0^q}{1+qh}\sup\bigg\{& \kappa\in\mathbb{R}:\ \exists \psi\in C_b^1\pr{\mathbb{R}^{m}} \textnormal{ such that } \forall \pr{\bar{x},\bar{u}}\in \mathbb{R}^{m}\times \bar{U}
,\\&\kappa\leq \mathcal{L}^{\bar{u}}\psi(\bar{x})+\phi(\bar{x})+(1+qh)\pr{\psi(\bar{x}_0)-\psi(\bar{x})},\bigg\}
\end{split}
 \end{split}
\end{equation}where the infinitesimal generator is given by 
\begin{equation}
\label{Lu}
\begin{split}
\mathcal{L}^{\bar{u}}\psi(\bar{x}):=\scal{f\pr{\bar{x},\bar{u}}
 ,\nabla \psi(\bar{x})}+\lambda\int_{\mathbb{R}}\pr{\psi\pr{\bar{x}+g\pr{\bar{x},\bar{u},y}}-\psi(\bar{x})}\mathbb{P}_{C_1}(dy).
 \end{split}
\end{equation}  
\item Furthermore, the test functions $\psi$ in the dual (supremum) formulation can be chosen lower bounded i.e.  \[\psi(\cdot)\geq \frac{1}{1+qh}\inf_{\bar{x}\in\mathbb{R}^{m}}\phi(\bar{x}). \]\end{enumerate}
 \end{theorem}
 The proof is essentially given in \cite[Corollary 8, Theorem 7,  and Section 3.3]{goreac_serea_ZubovPDMP}. For our readers' convenience, we will specify in the Appendix the exact references as well as the novelties (essentially in the $\mathcal{P}_2$ compactness). 
 \section{Relaxed Formulation for the Running Maximum Control Problem}\label{SecMain1}
 \subsection{Formulation and Main Results}
For non-degeneracy issues, let us fix $h>0$ large enough, and consider, for $\bar{x}_0\in\mathbb{R}^m$,
\begin{equation}
\label{V}
V(\bar{x}_0):=\inf_{\gamma\in \Theta_1(1,\bar{x}_0)}\norm{(a,\bar{x})\mapsto aL\pr{\bar{x}}}_{\mathbb{L}^\infty\pr{d\gamma}}.
\end{equation}
\begin{remark}
Note that this is a relaxation formulation for a value function of type
\[\inf_{\bar{u}\in\bar{\mathcal{U}}}\norm{\pr{\omega,t}\mapsto e^{-ht}L\pr{\bar{X}^{\bar{x}_0,\bar{u}}(\omega,t)}}_{\mathbb{L}^\infty\pr{\Omega\times\mathbb{R}_+,\mathbb{P}\otimes e^{-t}dt;\mathbb{R}}}.\]
\end{remark}
We are going to make the following assumption.
\begin{ass}\label{AssgenL}\begin{enumerate}
\item The cost function $L:\mathbb{R}^m\longrightarrow\mathbb{R}_+$ is bounded away from $0$ and upper-bounded by $\norm{L}_0<\infty$, and Lipschitz-continuous with Lipschitz constant $\pp{L}_1$.
\item The discount parameter satisfies
\[h\geq 1+\frac{3\pp{f}_1+\lambda\int_{\mathbb{R}}\pp{1+4\pp{g}_1(y)+2\pp{g}_1^2(y)}\mathbb{P}_{C_1}(dy)}{4}.\]
\end{enumerate}
\end{ass}
\begin{remark}
The denominator is $4$ since we are going to state the continuity results in the approximating $\mathbb{L}^q$ problems with $q\geq 2$. The condition can actually be given with a denominator $2q_0$ and the assertion remain valid with $q\geq q_0$. In particular, the asymptotic behavior does not change. From this point of view, one understands that, in view of our previous assumption that $\lambda\leq 1$,  the condition is actually $h>\max\set{\lambda,1},$ which is reasonable enough.
\end{remark}
We are now able to state the first main theoretical result of the paper providing a dual description for the value function $V$.
\begin{theorem}\label{ThMain}For every $\bar{x}_0\in\mathbb{T}_n\times \mathbb{R}$, the value function $V$ has a dual formulation given as follows.
\begin{equation}
\label{Vdual}
\begin{split}
V\pr{\bar{x}_0}=\sup \bigg\{ \psi(\bar{x}_0):\ &\psi\in C_b^1\pr{\mathbb{R}^{m};\mathbb{R}}, \exists q>1 \textnormal{ s. t. }\forall \pr{\bar{x},u}\in \mathbb{R}^{m},\times \bar{U},\\
0\leq &-\frac{1+hq+\lambda}{q}\psi(\bar{x})+\scal{f\pr{\bar{x},\bar{u}},\nabla \psi(\bar{x})}+\frac{1}{q}\pr{\frac{L\pr{\bar{x}}}{\psi\pr{\bar{x}}}}^q\psi(\bar{x})\\&+\frac{\lambda}{q}\int_{\mathbb{R}}\pp{\frac{\psi\pr{\bar{x}+g\pr{\bar{x},\bar{u},y}}}{\psi\pr{\bar{x}}}}^q\mathbb{P}_{C_1}(dy)\psi\pr{\bar{x}}\ \bigg\}.
\end{split}
\end{equation}
\end{theorem}

\begin{proof}The assertion follows from the $\sup$ characterization of $V$ given in Proposition \ref{PropV=supVq} and the dual characterizations of $V_q$ exhibited in Theorem \ref{PropTool1}.
\end{proof}

{We wish to note that the right-hand part of the inequality in \eqref{Vdual} is tailor-made in order to make $q\rightarrow\infty$. While we could have applied the characterization in \eqref{Eqinfsup} directly, this would lead to a maximization of $\kappa^{\frac{1}{q}}$ which heuristically converges to a trivial quantity, and involve, on the right-hand, $\phi=L^q$ which may diverge. It is for this reason that our characterization uses a slightly different formulation.}

The remaining of the section is devoted to the proof of Theorem \ref{ThMain} in several steps.
\subsection{Step 1. The $\mathbb{L}^q$ problems}\label{SubsecStep1}
As an essential maximum, it is usual to approximate $V$ via $\mathbb{L}^q$ problems  considering, for $q>0$, and $\bar{x}_0\in\mathbb{R}^m$,
\begin{equation}
\label{VN}
V_q(\bar{x}_0):=\inf_{\bar{u}\in\bar{\mathcal{U}}}\pr{\mathbb{E}\pp{\int_0^{\infty}e^{-(1+qh)t}L^q\pr{\bar{X}^{\bar{x}_0,\bar{u}}(t)}dt}}^{\frac{1}{q}}.
\end{equation}
In connection to these value functions, we introduce the integro-differential Hamilton-Jacobi equation
\begin{equation}
\label{HJBq}
\begin{cases}
&(1+qh)W(\bar{x})-L^q(\bar{x})+H\pr{\bar{x},\nabla W(\bar{x}),W}=0,\textnormal{ where}\\
&H(\bar{x},\theta,\psi):=\sup_{\bar{u}\in \bar{U}}\set{-\scal{f(\bar{x},\bar{u}),\theta}-\lambda\int_{\mathbb{R}}\pp{\psi\pr{\bar{x}+g\pr{\bar{x},\bar{u},y}}-\psi(\bar{x})}\mathbb{P}_{C_1}(dy)},
\end{cases}
\end{equation}
for $\theta\in\mathbb{R}^m$ and $\psi$ being a bounded measurable real-valued function on $\mathbb{R}^m$.

\begin{theorem}\label{PropTool1}
Let us enforce Assumptions \ref{Ass_gen} and \ref{AssgenL}. The following assertions hold true, for every $q\geq 2$.
\begin{enumerate}
\item $V_q^q$ is the unique bounded and uniformly continuous viscosity solution to \eqref{HJBq}. 
\item The sequence $(V_q)_{q\geq 2}$ is non-decreasing and bounded. Furthermore, $V_q\geq \inf_{\bar{x}\in\mathbb{R}^m}L(\bar{x})>0$.
\item The following relaxed and dual formulations are valid.\begin{equation}
\label{dualV_q}
\begin{split}
V_q(\bar{x}_0)=\inf\bigg\{&\norm{\pr{a,\bar{x}}\mapsto aL(\bar{x})}_{\mathbb{L}^q\pr{d\gamma}}:\  \gamma\in\Theta_1(1,\bar{x}_0)\bigg\}\\=\sup\bigg\{&\psi(\bar{x}_0):\ \psi\in C_b^1\pr{\mathbb{R}^{m};\mathbb{R}_+}, \textnormal{ s.t. }\forall (\bar{x},\bar{u})\in \mathbb{R}^m\times\bar{U},\\
&0\leq -\frac{1+hq+\lambda}{q}\psi(\bar{x})+\scal{f\pr{\bar{x},\bar{u}},\nabla \psi(\bar{x})}+\frac{1}{q}\pr{\frac{L\pr{\bar{x}}}{\psi\pr{\bar{x}}}}^q\psi(\bar{x})\\&+\frac{\lambda}{q}\int_{\mathbb{R}}\pp{\frac{\psi\pr{\bar{x}+g\pr{\bar{x},\bar{u},y}}}{\psi\pr{\bar{x}}}}^q\mathbb{P}_{C_1}(dy)\psi\pr{\bar{x}}\bigg\}.
\end{split}
\end{equation}
\end{enumerate}
\end{theorem}
\begin{proof}[Sketch of the proof of Theorem \ref{PropTool1}]
\begin{enumerate}
\item Let us only sketch the uniform continuity for $q\geq 2$ fixed in order to understand the lower bound assumption on $h$.  One begins with noting that $\abs{L^q(\bar{x})-L^q(\bar{y})}\leq q\norm{L}_0^{q-1}\pp{L}_1\abs{\bar{x}-\bar{y}}$.  Then, if $\bar{u}$ is a fixed admissible control and $\bar{x}_1,\bar{x}_2\in\mathbb{R}^m$, and the associated trajectories are denoted by $\bar{X}_k:=\bar{X}^{\bar{x}_k,\bar{u}}$, with $k\in\set{1,2}$, then, one gets
\begin{align*}
&\abs{\mathbb{E}\pp{\int_0^\infty e^{-(1+qh)t}L^q\pr{\bar{X}_1(t)}dt}-\mathbb{E}\pp{\int_0^\infty e^{-(1+qh)t}L^q\pr{\bar{X}_2(t)}}dt}\\&\leq q\norm{L}_0^{q-1}\pp{L}_1\int_0^\infty e^{-(1+qh)t}\mathbb{E}\pp{\abs{\bar{X}_1(t)-\bar{X}_2(t)}}dt\\
&\leq q\norm{L}_0^{q-1}\pp{L}_1\pr{\int_0^\infty e^{-t}dt\int_0^\infty e^{-(1+2qh)t}\mathbb{E}\pp{\abs{\bar{X}_1(t)-\bar{X}_2(t)}^2}dt}^{\frac{1}{2}}.
\end{align*}
The conclusion follows from the first estimate in Proposition \ref{PropEstimTraj}.\\
The link with the integro-differential equation is standard; see, for instance, \cite[Theorem 1.1]{Soner86_2} (or \cite{goreac_serea_ZubovPDMP}, or \cite{Serrano_2015}).
\item The second assertion follows from the fact that $L$ is bounded away from $0$, while the monotonicity property is standard from Hölder's inequality.
\item Let us now briefly explain the third assertion.
\begin{itemize}
\item The primal formulation is given in the second assertion of Theorem \ref{ThGaclco}.\\
\item If $\psi\in C_b^1\pr{\mathbb{R}^m;\mathbb{R}_+}$, and satisfies the inequality in the right-hand member of \eqref{dualV_q}, then
\begin{align*}
0\leq &-(1+hq)a^q\psi^q(\bar{x})+a^q\scal{f\pr{\bar{x},\bar{u}},\nabla \psi^q(\bar{x})}+a^qL^q\pr{\bar{x},\bar{u}}\\&+\lambda\int_{\mathbb{R}}a^q\pp{\psi^q\pr{\bar{x}+g\pr{\bar{x},\bar{u},y}}-\psi^q\pr{\bar{x}}}\mathbb{P}_{C_1}(dy),
\end{align*}for every $a>0$. By adding $\psi^q(\bar{x}_0)$ and integrating with respect to an arbitrary $\gamma\in \Theta\pr{1,\bar{x}_0}$ it follows, owing to the linear constraint (written for $\phi(a,\bar{x}):=a^q\psi^q(\bar{x})$), that 
\[\psi^q(\bar{x}_0)\leq \int_{\mathbb{R}^{m+1}\times\bar{U}}a^qL^q\pr{\bar{x}}\gamma(da,d\bar{x},d\bar{u}).\]As a consequence, using the first assertion,
$V_q\pr{\bar{x}_0}$ is no lower than the dual ($\sup$) formulation.
\item Let us now show the converse inequality. We define the \emph{shaken} value functions as follows. For $\bar{x}_0\in\mathbb{R}^m$ and $\varepsilon>0$, we let \begin{equation}
\label{v+}
V^{+,\varepsilon}_q(\bar{x}_0):=\inf_{\pr{\bar{u},e}\in\bar{\mathcal{U}}^+}\pr{\mathbb{E}\pp{\int_0^\infty e^{-(1+qh)t}L^q\pr{\bar{X}^{+,\bar{x}_0,\bar{u},\varepsilon e}_1(t)+\varepsilon e(t)}dt}}^{\frac{1}{q}}.
\end{equation}
The admissible controls set $\bar{\mathcal{U}}^+$ stands for predictable couples $(\bar{u},e)$ taking their values in $\bar{U}\times B_1(0)$. To simplify notations, we let \[v_q^{+,\varepsilon}:=\pr{V_q^{+,\varepsilon}}^q,\ v_q:=V_q^q.\]\\
Reasoning as we have already done for the first assertion,  but invoking the second estimate in Proposition \ref{PropEstimTraj}, it follows that 
\begin{equation}\label{dVepsV}
\sup_{\bar{x}_0\in\mathbb{R}}\abs{v^{+,\varepsilon}_q(\bar{x}_0)-v_q(\bar{x}_0\Db{)}}\leq C_q\varepsilon,
\end{equation}where $C_q$ is a constant depending on $q\geq 2$ but not on $\varepsilon>0$. The reader is invited to note that the assumption on $h$ is made exactly to cover the exponential parameter appearing in the second estimate in Proposition \ref{PropEstimTraj}.\\
The Hamilton-Jacobi integro-differential equation satisfied by $v_q^{+,\varepsilon}$ is
\begin{equation}
\label{HJB+}
\begin{split}
0=(1+hq)W(\bar{x})+\sup_{\bar{u}\in\bar{U},\ \abs{e}\leq \varepsilon}\big\{&-\scal{f\pr{\bar{x}+e,\bar{u}},\nabla W(\bar{x})}-L^q\pr{\bar{x}+e}\\&-\lambda\pp{\int_{\mathbb{R}}W\pr{\bar{x}+e+g\pr{\bar{x}+e,\bar{u},y}}\mathbb{P}_{C_1}(dy)-W\pr{\bar{x}}}\big\}
\end{split}
\end{equation}
Following the approach in the proof of \cite[Theorem 7]{goreac_serea_ZubovPDMP} (itself inspired by \cite{krylov_00}, and further spelled out in \cite{barles_jakobsen_02} for Brownian perturbed dynamics), one considers the convoluted function $v_{q,\varepsilon}^{+}:=v_q^{+,\varepsilon}\star\rho_\varepsilon$, where $\rho_\varepsilon(\cdot):=\frac{1}{\varepsilon^{m}}\rho\pr{\frac{y}{\varepsilon}}$ for $\rho$ a $C_\infty$-regular, non-negative function on $\mathbb{R}^{m}$ whose support is contained in the unit ball $B_1(0)$ and $\int_{\mathbb{R}^{m}}\rho(x)dx=1$.  The reader will easily note that \eqref{dVepsV} implies \begin{equation*}\
\abs{v^{+}_{q,\varepsilon}(\bar{x}_0)-v_q(\bar{x}_0\Db{)}}\leq C_q\varepsilon+\sup_{\abs{\bar{x}-\bar{x}_0}\leq \varepsilon}\abs{v_q(\bar{x})-v_q(\bar{x}_0)},
\end{equation*}
such that \begin{equation}\label{ddVepsV}
\lim_{\varepsilon\rightarrow 0+}\abs{v^{+}_{q,\varepsilon}(\bar{x}_0)-v_q(\bar{x}_0\Db{)}}=0.
\end{equation}
These functions provide regular subsolutions of \eqref{HJBq}, i.e. , for all $\bar{u}\in\bar{U}$, and all $\bar{x}\in\mathbb{R}^{m}$,
\begin{equation}
\label{SubsolHJB}
\begin{split}
0\geq &(1+hq)v_{q,\varepsilon}^{+}(\bar{x})-\scal{f\pr{\bar{x},\bar{u}},\nabla v_{q,\varepsilon}^{+}(\bar{x})}-L^q\pr{\bar{x}}\\&-\lambda\pp{\int_{\mathbb{R}}v_{q,\varepsilon}^{+}\pr{\bar{x}+g\pr{\bar{x},\bar{u},y}}\mathbb{P}_{C_1}(dy)-v_{q,\varepsilon}^{+}\pr{\bar{x}}}.
\end{split}
\end{equation}
The regularity of these functions, and the fact that they are bounded away from $0$ allows one to write the equation for \[V_{q,\varepsilon}^{+}:=\pp{v_{q,\varepsilon}^{+}}^{\frac{1}{q}},\]which reads
\begin{equation}
\label{SubsolHJB'}
\begin{split}
0\leq &-\frac{1+hq+\lambda}{q}V_{q,\varepsilon}^{+}(\bar{x})+\scal{f\pr{\bar{x},\bar{u}},\nabla V_{q,\varepsilon}^{+}(\bar{x})}+\frac{1}{q}\pr{\frac{L\pr{\bar{x}}}{V_{q,\varepsilon}^{+}\pr{\bar{x}}}}^qV_{q,\varepsilon}^{+}(\bar{x})\\&+\frac{\lambda}{q}\int_{\mathbb{R}}\pp{\frac{V_{q,\varepsilon}^{+}\pr{\bar{x}+g\pr{\bar{x},\bar{u},y}}}{V_{q,\varepsilon}^{+}\pr{\bar{x}}}}^q\mathbb{P}_{C_1}(dy)V_{q,\varepsilon}^{+}\pr{\bar{x}}.
\end{split}
\end{equation}
It follows that $V_{q,\varepsilon}^{+}$ can be used as a test function in the right-hand member of \eqref{dualV_q} and, thus, owing to \eqref{ddVepsV},
\begin{align*}V_q(\bar{x}_0)=\lim_{\varepsilon\rightarrow 0+}V_{q,\varepsilon}^{+}(\bar{x}_0)\leq \sup\bigg\{&\psi(\bar{x}_0):\ \psi\in C_b^1\pr{\mathbb{R}^{m};\mathbb{R}_+}, \textnormal{ s.t. }\forall (\bar{x},\bar{u})\in \mathbb{R}^m\times\bar{U},\\
&0\leq -\frac{1+hq+\lambda}{q}\psi(\bar{x})+\scal{f\pr{\bar{x},\bar{u}},\nabla \psi(\bar{x})}+\frac{1}{q}\pr{\frac{L\pr{\bar{x}}}{\psi\pr{\bar{x}}}}^q\psi(\bar{x})\\&+\frac{\lambda}{q}\int_{\mathbb{R}}\pp{\frac{\psi\pr{\bar{x}+g\pr{\bar{x},\bar{u},y}}}{\psi\pr{\bar{x}}}}^q\mathbb{P}_{C_1}(dy)\psi\pr{\bar{x}}\bigg\}.
\end{align*}
\end{itemize}
\end{enumerate}
\end{proof}
\begin{remark}\label{RemgPartCase}
We introduce the following.
\begin{ass}\label{Ass_g_1}
The function $\bar{x}\mapsto \bar{x}+g\pr{\bar{x},\bar{u},y}$ is $1$-Lipschitz, for all $y$ in the support of $\mathbb{P}_{C_1}$, and all controls $\bar{u}\in \bar{U}$,  and \[h\geq 2\pp{f}_1+\lambda.\]
\end{ass}
\Db{This condition plays a pivotal role in establishing classical links between the limit value functions and a class of Hamilton-Jacobi integro-differential equations, which align closely with the framework in \cite{barron_ishii_89}. For further details, we refer to the uniformity condition specified in Proposition \ref{PropBI}.}
Under the aforementioned assumption, we have the following sharper results
\begin{enumerate}
\item The family \[\set{V_{q,\varepsilon}^+,V_q^{+,\varepsilon}:\ q\geq 2,\varepsilon>0}\] is equi-continuous (more precisely, $\pp{L}_1$-Lipschitz continuous), and, therefore relatively compact in the set of bounded uniformly continuous functions.\\
Indeed, in this case, having fixed $\bar{x}_k\in\mathbb{R}^m$ and denoting, for arbitrary controls $\pr{\bar{u},e}$, $\bar{X}_k^+$ the associated trajectories ($k\in\set{1,2}$), one has, \Db{owing to the last assertion in Proposition \ref{PropEstimTraj}, i.e., to the inequality \eqref{EstimX1-X2_q}},
\begin{align*}
&\abs{\pr{\mathbb{E}\pp{\int_0^\infty e^{-(1+qh)t}L^q\pr{\bar{X}^{+}_1(t)+e(t)}dt}}^{\frac{1}{q}}-\pr{\mathbb{E}\pp{\int_0^\infty e^{-(1+qh)t}L^q\pr{\bar{X}^{+}_2(t)+e(t)}dt}}^{\frac{1}{q}}}\\
&\leq\int_0^\infty e^{-(1+qh)t} \mathbb{E}\pp{\abs{L\pr{\bar{X}^{+}_1(t)+e(t)}-L\pr{\bar{X}^{+}_2(t)+e(t)}}^q}dt^{\frac{1}{q}}\\
&\leq \pp{L}_1\pr{\int_0^\infty e^{-(1+qh)t}e^{q\pp{f}_1t}dt}^{\frac{1}{q}}\abs{\bar{x}_1-\bar{x}_2}\leq  \pp{L}_1\abs{\bar{x}_1-\bar{x}_2},
\end{align*}as soon as $h\geq \pp{f}_1$. This implies the fact that $V_q^{+,\varepsilon}$ are $\pp{L}_1$-Lipschitz continuous and the same is valid for $V_{q,\varepsilon}^+$.
\item Furthermore,  there exists a constant $C>0$ independent of $q\geq 2$ and of $\varepsilon>0$ such that \begin{equation}\label{EqRem2}
\sup_{q\geq 2}\sup_{\bar{x}_0\in\mathbb{R}^m}\abs{V_q^{+,\varepsilon}(\bar{x}_0)-V_q(\bar{x}_0)}\leq C\varepsilon.
\end{equation}
Indeed,  with the notation $\bar{X}=\bar{X}^{\bar{x}_1,\bar{u}}$,  for $e$ taking its values in $B_\varepsilon(0)$, and owing to the inequality \eqref{EstimX+-xg}, it follows that
\begin{align*}
&\abs{\pr{\mathbb{E}\pp{\int_0^\infty e^{-(1+qh)t}L^q\pr{\bar{X}_1^{+}(t)+e(t)}dt}}^{\frac{1}{q}}-\pr{\mathbb{E}\pp{\int_0^\infty e^{-(1+qh)t}L^q\pr{\bar{X}(t)}dt}}^{\frac{1}{q}}}\\
&\leq\pr{\int_0^\infty e^{-(1+qh)t} \mathbb{E}\pp{\abs{L\pr{\bar{X}^{+}_1(t)+e(t)}-L\pr{\bar{X}(t)}}^q}dt}^{\frac{1}{q}}\\
&\leq \pp{L}_1\pp{\int_0^\infty e^{-(1+qh)t}\pr{q\mathbb{E}\pp{\abs{(\bar{X}_1^+-\bar{X})(t)}^q}+q\varepsilon^q}dt}^{\frac{1}{q}}\\
&\leq \pp{L}_1\varepsilon\pp{q\pr{\pp{f}_1+\lambda q+1}\pr{\int_0^\infty e^{-(1+qh)t}\pp{e^{\pr{(2q-1)\pp{f}_1+\lambda q}t}t+1}dt}}^{\frac{1}{q}}.
\end{align*}The conclusion follows as soon as $h\geq 2\pp{f}_1+\lambda$.
\item By the first assertion,
\begin{align*}
\abs{V_{q,\varepsilon}^+(\bar{x}_0)-V_q^{+,\varepsilon}(\bar{x}_0)}&=\abs{\pr{\int_{\mathbb{R}^m}\pr{V_q^{+,\varepsilon}}^q(\bar{x}_0-\bar{y})\rho_\varepsilon(\bar{y})d\bar{y}}^{\frac{1}{q}}-\pr{\int_{\mathbb{R}^m}\pr{V_q^{+,\varepsilon}}^q(\bar{x}_0)\rho_\varepsilon(\bar{y})d\bar{y}}^{\frac{1}{q}}}\\&\leq \pr{\int_{\mathbb{R}^m}\abs{V_q^{+,\varepsilon}(\bar{x}_0-\bar{y})-V_q^{+,\varepsilon}(\bar{x}_0)}^q\rho_\varepsilon(\bar{y})d\bar{y}}^{\frac{1}{q}}\leq \pp{L}_1\varepsilon.
\end{align*}It follows that  
\begin{equation}\label{eqRem3}
\sup_{q\geq 2}\sup_{\bar{x}_0\in\mathbb{R}^m}\abs{V_{q,\varepsilon}^+(\bar{x}_0)-V_q(\bar{x}_0)}\leq (C+\pp{L}_1)\varepsilon.
\end{equation}.
\end{enumerate}
\end{remark}

\subsection{Step 2. The Limit as $q\rightarrow\infty$ and Link to the Value Function \eqref{VN}}\label{SubsecStep2}
We have the following natural characterization of $V$.
\begin{proposition}\label{PropV=supVq}
 For every $\bar{x}_0\in\mathbb{R}^m$,  the following characterization holds true
 \begin{equation}
 V(\bar{x}_0):=\inf_{\gamma\in\Theta_1(1,\bar{x}_0)}\norm{(a,\bar{x})\mapsto aL\pr{\bar{x}}}_{\mathbb{L}^\infty\pr{d\gamma}}=\sup_{q>1}V_q\pr{\bar{x}_0}.
 \end{equation}
\end{proposition}
\begin{proof}
Since $\pp{0,1}\times \mathbb{R}^m\ni(a,\bar{x})\mapsto a^qL^q\pr{\bar{x}}$ are bounded continuous for every $q\geq 2$,  the linearized formulation of $V_q$ (cf. Theorem \ref{ThGaclco}, assertion 2) yields \[V_q\pr{\bar{x}_0}=\inf_{\gamma\in \Theta_1\pr{1,\bar{x}_0}}\pp{\int_{\mathbb{R}^{m+1}}a^qL^q(\bar{x})\gamma(da,d\bar{x})}^{\frac{1}{q}}\leq V(\bar{x}_0).\]Owing to the compactness of $\Theta_1\pr{1,\bar{x}_0}$,  there exists an optimal $\gamma_q\in \Theta_1\pr{1,\bar{x}_0}$ such that \[V_q\pr{\bar{x}_0}=\int_{\mathbb{R}^{m+1}}a^qL^q(\bar{x})\gamma_q(da,d\bar{x}).\]
Again by the compactness of the family $\Theta_1\pr{1,\bar{x}_0}$, there exists a subsequence still denoted $\gamma_q$ weakly convergent (and also $\mathcal{W}_2$-convergent) to $\gamma_\infty\in \Theta_1\pr{1,\bar{x}_0}$.  We claim that \begin{equation}
\label{gaminf}
\sup_{q>1}V_q\pr{\bar{x}_0}\geq \norm{(a,\bar{x})\mapsto aL\pr{\bar{x}}}_{\mathbb{L}^\infty\pr{d\gamma_\infty}}\geq V(\bar{x}_0),
\end{equation}which will conclude our proof. To this purpose, let us argue by contradiction and assume that there exists $\delta>0$ such that for every $q\geq 2$,
\[V_q\pr{\bar{x}_0}<-\delta+\norm{(a,\bar{x})\mapsto aL\pr{\bar{x}}}_{\mathbb{L}^\infty\pr{d\gamma_\infty}}.\]
For $q_0>1$ fixed,  and every $q\geq q_0$,  the choice of $\gamma_{q}$ yields
\begin{align*}
\pp{\int_{\mathbb{R}^{m+1}}a^{q_0}L^{q_0}(\bar{x})\gamma_q(da,d\bar{x})}^{\frac{1}{q_0}}&\leq \pp{\int_{\mathbb{R}^{m+1}}a^{q}L^{q}(\bar{x})\gamma_q(da,d\bar{x})}^{\frac{1}{q}}\\&=
V_{q}\pr{\bar{x}_0}<-\delta+\norm{(a,\bar{x})\mapsto aL\pr{\bar{x}}}_{\mathbb{L}^\infty\pr{d\gamma}_\infty}.\end{align*}By taking $q\rightarrow\infty$ and due to the continuity and boundedness of the cost function, one gets 
\begin{align*}
\pp{\int_{\mathbb{R}^{m+1}}a^{q_0}L^{q_0}(\bar{x})\gamma_\infty(da,d\bar{x})}^{\frac{1}{q_0}}\leq -\delta+\norm{(a,\bar{x})\mapsto aL\pr{\bar{x}}}_{\mathbb{L}^\infty\pr{d\gamma_\infty}}.\end{align*}
By passing $q_0\rightarrow \infty$, we obtain a contradiction.
\end{proof}

\section{A Classical Formulation}\label{SecClassSol}
In order to simplify the framework, let us enforce Assumption \eqref{Ass_g_1} throughout the section.
The classical formulation requires investigating the (formal) limit in the equation satisfied by the functions $V_q$ i.e. 
\begin{equation}\label{HJBeq'}
\begin{split}
0=\frac{1+hq+\lambda}{q}\psi(\bar{x})+\sup_{\bar{u}\in\bar{U}}\bigg\{&-\scal{f\pr{\bar{x},\bar{u}},\nabla \psi(\bar{x})}-\frac{1}{q}\pr{\frac{L\pr{\bar{x}}}{\psi\pr{\bar{x}}}}^q\psi(\bar{x})\\&-\frac{\lambda}{q}\int_{\mathbb{R}}\pp{\frac{\psi\pr{\bar{x}+g\pr{\bar{x},\bar{u},y}}}{\psi\pr{\bar{x}}}}^q\mathbb{P}_{C_1}(dy)\psi\pr{\bar{x}}\bigg\}
\end{split}
\end{equation}
 and, in particular, of the Hamiltonian part. To this aim, we begin with the following simple result which constitutes a slight generalization to \cite[Proposition 2.1]{barron_ishii_89}.
\begin{proposition}\label{PropBI}
Let $\phi_q:\mathbb{R}\longrightarrow\mathbb{R}_+$ be a non-decreasing sequence of functions converging uniformly on compact sets to $\phi_\infty$ and $\alpha$ be non-negative and continuous on the compact set $\bar{U}$ and $\beta$ be uniformly continuous and bounded on $\bar{U}\times \mathbb{R}_+$.  Furthermore, assume that\[{\inf}_{r\in\beta\pr{\bar{U}\times \mathbb{R}_+}}\phi_1(r)>1.\]Then, for every $\mu, \nu\in\mathbb{R}_+^*$, and every $\bar{x}$ fixed, 
\begin{equation}
\label{LimH}
\lim_{q\rightarrow\infty}\min_{\bar{u}\in\bar{U}}\pp{\alpha(\bar{u})+\frac{1}{q}\pr{\frac{\norm{\phi_q\pr{\beta\pr{\bar{u},\cdot}}}_{\mathbb{L}^q\pr{\mathbb{P}_{C_1}}}}{\nu}}^q\mu}=\inf\set{\alpha(\bar{u}):\ \bar{u}\in \bar{U}(\phi_\infty,\nu)},
\end{equation}
where \begin{equation}
\label{barUnu}
\bar{U}(\phi_\infty,\nu)=\set{\bar{u}\in \bar{U}:\ \phi_\infty\pr{\beta(\bar{u},y)}\leq \nu,  \mathbb{P}_{C_1}(dy)-a.s. }.
\end{equation}
\end{proposition}
The proof is quite standard, one inequality being obvious and the remaining one being obtained through contradiction reasoning. For our readers' sake, its proof is given in the Appendix. As usual, infimum over empty sets is set to be $+\infty$.
\begin{remark}
If the support of $\mathbb{P}_{C_1}$ is compact (or, more generally, by assuming \[\bar{u}\mapsto esssup\ \phi_\infty(\beta\pr{\bar{u},\cdot})\] to be continuous), non-emptiness of $\bar{U}(\phi_\infty,\nu)$ is guaranteed for $\nu\geq \underset{\bar{u}\in\bar{U}}{\inf}esssup\ \phi_\infty(\beta\pr{\bar{u},\cdot})$. The $ess sup$ is intended with respect to $\mathbb{P}_{C_1}$.
\end{remark}
Throughout the remaining of the section, we further assume that the domain of interest is some compact set $K$\footnote{for epidemics, this is reduced to the tetrahedron $\set{(s,i,r)\in\mathbb{R}^3_+\ :\ s+i+r\leq 1}$ and a normalization can be imposed on the capital by changes of numeraire}. We assume that the compact set $K$ is invariant with respect to the control system \eqref{SDEgen}. In particular, one can modify the dynamics on a compact set $K\subset K^0$ such that $f(\bar{x},\bar{u})=g(\bar{x},\bar{u},\cdot)=0$ for every $\bar{x}\in \mathbb{R}^m\setminus K^0$ and every $\bar{u}\in  \bar{U}$. This guarantees that, under the Assumption \eqref{Ass_g_1}, the convergence of $V_q$ to $V$ is uniform on $K^0$, hence on  $\mathbb{R}^m$.
\subsection{The HJB Equation}
We consider the following non-local Hamiltonian
\begin{equation}\begin{split}
H(\psi,x,r,p)&=\inf_{\bar{u}\in \bar{U}(\psi,x,r)}\scal{p,f(x,\bar{u})},\\&\textnormal{ where }\bar{U}(\psi,x,r)=\set{\bar{u}\in \bar{U}:\ \psi(x+g(x,\bar{u},y))\leq r,\ \mathbb{P}_{C_1}-a.s.},\end{split}
\end{equation}
for $x\in\mathbb{R}^m,\ r\in\mathbb{R},\ p\in\mathbb{R}^m$ and $\psi$ a bounded uniformly continuous function (shortly expressed as $\psi  \in BUC(\mathbb{R}^m;\mathbb{R}$)
and the equation
\begin{equation}\label{HJB}
-h\psi(x)+H(\psi,x,\psi(x),\nabla\psi(x))=0.
\end{equation}
Before providing the rigorous definition of solutions, let us give the following result on the relaxed Hamiltonians and the control sets. This is a straightforward generalization to the integro-differential case of \cite[Lemma 15]{GHS_2021}, compiling the results in \cite[Propositions 2.1 and 2.5, and Lemma 2.4]{barron_ishii_89}.
\begin{proposition}\label{PropU}
Let $\bar{U}(\psi,x,r)=\set{\bar{u}\in \bar{U}:\ \psi(x+g(x,\bar{u},y))\leq r,\ \mathbb{P}_{C_1}-a.s.}$. Then
\begin{enumerate}
\item If $r\leq r'$, then $\bar{U}(\psi,x,r)\subseteq \bar{U}(\psi,x,r')$.
\item If $r<r'$ and $(\psi_q)_q$ are bounded and equi-continuous converging uniformly on compact sets to $\psi=\psi_\infty$ as $q\rightarrow\infty$, then, for $q_0$ large enough, there exists $\delta(r,r',q_0)>0$ such that
\[\bar{U}(\psi_q,x',r)\subseteq \bar{U}(\psi_{q'},x,r'),\ \forall x'\in B_\delta(x), \ \forall q,q'\in\pp{q_0,\infty}.\]
\item Recall that $H(\psi, x,r,p)=\min_{\bar{u}\in \bar{U}(\psi,x,r)}\scal{p,f(x,\bar{u})}$. Then 
\begin{enumerate}
\item If $r\leq r'$, then $H(\psi, x,r,p)\geq H(\psi, x,r',p)$;
\item If $r<r'$, then, under the assumption of 2. above, for every $q_0>0$ large enough, there exists some $\delta(r,r',q_0)>0$ such that 
\[H(\psi_q,x',r,p)\geq H(\psi_{q'},x,r',p),\ \forall x'\in B_\delta(x),\textnormal{ and }\forall q,q'\in\pp{q_0,\infty}.\]
\item The relaxed Hamiltonians satisfy
\begin{equation}\begin{split}
\begin{cases}
H^*(\psi,x,r,p)&=\underset{\varepsilon\rightarrow 0}{\lim\sup}\set{H(\psi,x',r',p')\ : \abs{x'-x}+\abs{r'-r}+\abs{p'-p}\leq \varepsilon}\\&=H(\psi,x,r-,p)\\
H_*(\psi,x,r,p)&=\underset{\varepsilon\rightarrow 0}{\lim\inf}\set{H(\psi,x',r',p')\ : \abs{x'-x}+\abs{r'-r}+\abs{p'-p}\leq \varepsilon}\\&=H(\psi,x,r+,p).
\end{cases}
\end{split}
\end{equation}.
\end{enumerate}
\end{enumerate}
\end{proposition}
The proof is a straight-forward adaptation of the original ones in \cite{barron_ishii_89} and it will be sketched in the Appendix.
\subsection{Viscosity Solution}
We begin with providing the following definition.
\begin{definition}\label{DefSol}
\begin{enumerate}
\item An upper semi-continuous function $\psi$ is called a \emph{viscosity subsolution} of \eqref{HJB} if for every $\psi\in C^1$ and $x\in\arg\max\ (\psi-\phi)$,
\[-h\phi(x)+H^*(\phi,x,\psi(x),\nabla\phi(x))\geq 0.\]
\item A lower semi-continuous function $\psi$ is called a \emph{viscosity supersolution} of \eqref{HJB} if for every $\psi\in C^1$ and $x\in\arg\min\ (\psi-\phi)$,
\[-h\phi(x)+H_*(\phi,x,\psi(x),\nabla\phi(x))\leq 0.\]
\end{enumerate}
\end{definition}By taking a look at the proofs, the reader will easily note that the test functions $\psi$ in the previous definition can actually be taken bounded from bellow away from $0$. 
\begin{theorem}\label{ThClassSol}
Let us assume that there exists a compact set $K$ invariant with respect to the dynamics \eqref{SDEgen} and that $K\subset K^0$ is another compact set outside of which $f$ and $g$ are null. Furthermore, we assume that \ref{Ass_g_1} holds true.
Then, the function $V$ is the unique bounded $\pp{L}_1$-Lipschitz continuous viscosity solution in the sense of Definition \ref{DefSol}. It further satisfies the lower barrier $V\geq L$.
\end{theorem}
\Db{Concerning the argument in the proof,} neither the viscosity property, nor the uniqueness in the class of BUC functions are \Db{fundamentally} different from \cite[Theorem 16, Proposition 19]{GHS_2021} (themselves quasi-identical to the original results in \cite{barron_ishii_89}). All these results are based on the properties of the constrained controls and the relaxed Hamiltonians in Proposition \ref{PropU} combined with Proposition \ref{PropBI}. \Db{The main elements of proof are presented in the Appendix.}
\section{Going back to the reference model in Section 2}\label{SecNew}
{\subsection*{Positiveness of the Components of \eqref{SIR_dyn}}
For simplicity, we deal with the policyholder $j=1$ and will drop the dependency on $\bar{u}=(u,p)$. Note that $d(S+I+R)_1(t)=0$, which means that one preserves the total population to $(S+I+R)_1(0)=1$. Between the jumps, we deal with a classical SIR model and the components are known to be non-negative.\\
At jumping times, $S_1$ jumps to $\min\set{S_1, 1-\frac{I^0}{p_1}}$ and $I_1$ to $\max\set{I_1,\frac{I^0}{p_1}}$. Since $p_1\geq 1$, both components remain non-negative.
If, prior to the jump, $I_1<\frac{I^0}{p_1}$, it does not get updated, and it follows that $I_1+\min\set{S_1, 1-\frac{I^0}{p_1}}\leq I_1+S_1\leq1.$
Otherwise, $I_1$ updates to $\frac{I^0}{p_1}$, but 
$\frac{I^0}{p_1}+\min\set{S_1, 1-\frac{I^0}{p_1}}\leq 1.$
By construction, $R_1$ jumps to satisfy $I_1+S_1+R_1=1$, thus remaining non-negative. An alternative argument based on viability can be presented as in \cite[Proof of Theorem 1]{GLX_2022}.}
\subsection*{Induced Assumptions}\label{SubsfgIns}
For simplicity reasons, let us denote by $\bar{x}=(s,i,x)\in \mathbb{R}^{2n+1}$ (actually one only needs to consider $(s_j,i_j)\in \mathbb{T}$, for every $1\leq j\leq n$) and $\bar{u}=(u,p)$. The set of admissible control $\bar{\mathcal{U}}$ is the family of predictable $\bar{U}$-valued processes. The drift and jump coefficients $f$ and $g$ are explicitly written as vector columns with dimension $2n+1$ by setting
\begin{align*}g\pr{\bar{x},\bar{u},y}=g(s,i,x,u,p,y)&:=
\begin{pmatrix}
\pp{-\pr{s_j-1+\frac{\sum_{1\leq \Db{k}\leq n}i_{\Db{k}}}{np_j}}^+}_{ 1\leq j\leq n}\\
\pp{\pr{\frac{\sum_{1\leq \Db{k}\leq n}i_{\Db{k}}}{np_{j}}-i_j}^+}_{1\leq j\leq n}\\
-\sum_{1\leq k'\leq n}\pr{\frac{\sum_{1\leq k\leq n}i_k}{np_{k'}}-i_{k'}}^+uy
\end{pmatrix},\\ f\pr{\bar{x},\bar{u}}=f(s,i,x,u,p)&:=\begin{pmatrix}\pp{-\beta_jus_ji_j}_{1\leq j\leq n}\\ \pp{\pr{\beta_jus_j-\gamma_j}i_j}_{1\leq j\leq n}\\c(s,i,u,p)\end{pmatrix}.\end{align*}

\begin{ass}\label{Ass1}
\begin{enumerate}
\item[(i)] The set of protection levels $\mathcal{P}rev\subset [1,\infty)^n$ is compact.
\item[(ii)] The law $\mathbb{P}_{C_1}$ admits a finite fourth-order moment.
\item[(iii)] The premium coefficient $c:\mathbb{R}^{2n}\times U\times \mathcal{P}\rightarrow \mathbb{R}$ is Lipschitz continuous in the state variable $(s,i)$ uniformly with respect to the control parameters $(u,p)\in U\times \mathcal{P}rev$.
\end{enumerate}
\end{ass}
For a generic function $\phi:\mathbb{R}^{2n+1}\times U\times \mathcal{P}rev\rightarrow \mathbb{R}^{n'}$ (taking its values in a Euclidean space $\mathbb{R}^{n'}$,  we set
\begin{align}
\label{f0f1}
\begin{cases}
\norm{\phi}_0:=\sup\bigg\{&\abs{\phi(s,i,x,u,p)}:\ s=(s_1,\ldots, s_n),i=\pr{i_1,\ldots,i_n},\ \pr{s_j,i_j}\in \mathbb{T},\  1\leq j\leq n, \\& x\in\mathbb{R},\ u\in U, p\in\mathcal{P}rev \bigg\}\\
\pp{\phi}_1:=\sup\bigg\{&\frac{\abs{\phi(s,i,x,u,p)-\phi(s',i',x',u,p}}{\abs{\pr{s-s',i-i',x-x'}}}:\\ &s=(s_1,\ldots, s_n),i=\pr{i_1,\ldots,i_n},\ \pr{s_j,i_j}\in \mathbb{T},\  1\leq j\leq n, \\&s'=(s_1',\ldots, s_n'),i'=\pr{i_1',\ldots,i_n'},\ \pr{s_j',i_j'}\in \mathbb{T},\  1\leq j\leq n, \ x,x'\in\mathbb{R},\\ &(s,i,x)\neq (s',i',x'),\ u\in U, p\in\mathcal{P}rev \bigg\}.
\end{cases}
\end{align}
\begin{remark}\label{RemCoeff}With these notation and under the Assumption \ref{Ass1},  we have the following 
\begin{enumerate}
\item[(i)] The premium coefficient $c$ can be taken to be bounded since the dynamics only concern $\pr{s,i}$ such that $\pr{s_j,i_j}\in\mathbb{T}$ for all $1\leq j\leq n$.
\item[(ii)] The coefficient $f$ is bounded and Lipschitz continuous with $\norm{f}_0+\pp{f}_1<\infty$.
\item[(iii)] Similarly, $g\pr{\cdot,y}$ is bounded and Lipschitz continuous for every $y\in \mathbb{R}$ and the quantities $\norm{g\pr{\cdot,y}}_0$ and $\pp{g\pr{\cdot,y}}_1$ as before satisfy $\norm{g\pr{\cdot,y}}_0\leq \norm{g}_0\abs{y}$, $\pp{g\pr{\cdot,y}}_1\leq \pp{g}_1\abs{y}$, for constants $\norm{g}_0,\pp{g}_1\in\mathbb{R}_+$.
\item[(iv)] These assumptions can made for $f$ and $g$ for $(s,i)\in \mathbb{R}^{2n}$ instead of just restricted to $\mathbb{T}_n$ by simply extending the functions via projection of $(s,i)$ onto $\mathbb{T}_n$, or simply by setting $f$ and $g$ to $0$ outside a compact set containing $\mathbb{T}_n$.
\end{enumerate}
\end{remark}
\subsection*{Example of Premium}
Let us conclude the section with the analogous of a \emph{net premium}, i.e. an explicit form for $c$
\begin{proposition}\label{PropPremium}
Under the Assumption \ref{Ass1}, if $c=c_0$ is such that the velocity \[\lim_{t\rightarrow 0+}\frac{1}{t}\mathbb{E}\pp{\abs{X^{S_0,I_0,x,u,p}(t)-x}}=0, \]\Db{for a piecewise constant policy $(u,p)$}, then, \begin{equation}\label{Expc0}c_0(s,i,u,p)=\sum_{1\leq j\leq n}\lambda\mathbb{E}\pp{C_1}\Db{u}\pr{\frac{\sum_{1\leq \Db{k} \leq n}i_{\Db{k}}}{np_j}-i_j}^+.\end{equation}
\end{proposition}
The proof relies on standard estimates involving Poisson processes and related smooth equations. We shall sketch the main arguments at the end of the Appendix.
\begin{remark}
\begin{enumerate}
\item The form in \eqref{Expc0} shows that the premium only depends on \emph{infectiousness} and \emph{exposure} via the protection levels.
\item The function $c_0$ cannot be obtained as $\sum_{1\leq j\leq n}c_{0,j}(i_j,p_j)$, for non-negative functions $c_{0,j}$. Indeed, $c_0(i,i,\ldots, i,p)=0$, for every $i\in \pp{0,1}$ and every $p$. As a by-product, this would imply that $c_{0,j}=0$ or, again, $c_0=0$. This is to be interpreted as the fact that the insurer does not take $n$ individual clients, but a whole network with $n$ edges.
\end{enumerate}
\end{remark}
{
\subsection{Limits of the Model}
\textbf{Network Topology Beyond Star-Shaped Model:} While the star topology offers a clean, tractable starting point for cyber-risk analysis, extending methods to general graphs is essential to capture the complexity and interdependencies of real-world networks, thereby improving practical risk management and resilience strategies. Two elements are worth mentioning.\\
1. The more realistic models would require \emph{attack graphs} capturing multiple attack paths providing better dynamic model (and multi-node aspects). \\
2. While our model uses spreads from the average infectiousness, the \emph{directed acyclic graphs} can account for cascading effects. From this dynamical point of view, \cite{GoreacKobylanskiMartinez_2016} provides a more realistic framework. \\
\noindent 
\textbf{More Realistic Epidemics Models:} Within each edge, we have used a simple SIR model, and have dropped the $R$ component in order to simplify the model.  Of course,  loss of acquired immunity or reinfections may occur in a more direct way (i.e., possibly non-jump transitions from $R$ to $S$ or from $R$ to $I$), accounting for more realistic models. \\
Furthermore, standard SIR models can underestimate epidemic severity and reproduction numbers by failing to capture delays and heterogeneity in multi-wave or spatially staggered outbreaks. This bias may lead to inaccurate predictions and under-preparedness, highlighting the need for model extensions that explicitly account for temporal and spatial epidemic variations. 
}
\section{Appendix}
\subsection{Proof of Proposition \ref{PropEstimTraj}}
\begin{proof}[Proof of Proposition \ref{PropEstimTraj}]
\begin{enumerate}
\item 
Let us fix a predictable control $(\bar{u},e)$ and the initial datum. Itô's formula applied to $\abs{\cdot}^2$ and the process $\bar{X}^{+}_1-\bar{X}^{+}_2:=\bar{X}^{+,\bar{X}_1,\bar{u},e}-\bar{X}^{+,\bar{X}_2,\bar{u},e}$ on $\pp{0,t}$ yields
\begin{equation*}
\begin{split}
&\mathbb{E}\pp{\abs{\pr{\bar{X}_1^+-\bar{X}_2^+}(t)}^2}\\
&=\abs{\pr{\bar{X}_1-\bar{X}_2}(0)}^2\\&\quad+2\int_0^t\mathbb{E}\pp{\scal{f^+\pr{\bar{X}_1^+(l),\bar{u}(l),e(l)}-f^+\pr{\bar{X}^+_2(l),\bar{u}(l),e(l)},\bar{X}_1^+(l)-\bar{X}^+_2(l)}dl}\\&\quad+\int_0^t\begin{split}&\mathbb{E}\bigg[\abs{\bar{X}_1^+(l-)+g^+\pr{\bar{X}_1^+(l-),\bar{u}(l),e(l),y}-\bar{X}_2^+(l-)-g^+\pr{\bar{X}_2^+(l-),\bar{u}(l),e(l),y}}^2\\&-\abs{\bar{X}_1^+(l-)-\bar{X}_2^+(l-)}^2\bigg]\lambda\mathbb{P}_{C_1}(dy)
\end{split}\\
&\leq \mathbb{E}\pp{\abs{\pr{\bar{X}^+_1-\bar{X}^+_2}(0)}^2}\\&\quad+\pr{2\pp{f}_1+\lambda\int_{\mathbb{R}}\pp{2\pp{g}_1\pr{y}+\pp{g}_1^2(y)}\mathbb{P}_{C_1}(dy)}\int_0^t\mathbb{E}\pp{\abs{\pr{\bar{X}^+_1-\bar{X}^+_2}(l)}^2}dl.
\end{split}\end{equation*}
The conclusion follows owing to Gronwall's inequality.
\item Similarly, let us fix $\bar{u}$ and $e$ and the initial datum $\bar{X}_0\in\mathbb{R}^m$. As before, we make a notation by setting $\bar{X}^+=\bar{X}^{+,\bar{X}_0,\bar{u},e},\ \bar{X}=\bar{X}^{\bar{X}_0,\bar{u}}$.  Itô's formula for the squared norm applied to the difference of the two processes yields
\begin{equation*}
\begin{split}
&\mathbb{E}\pp{\abs{\pr{\bar{X}^+-\bar{X}}(t)}^2}\\
&=2\int_0^t\mathbb{E}\pp{\scal{f\pr{\bar{X}^+(l)+e(l),\bar{u}(l)}-f\pr{\bar{X}(l),\bar{u}(l)},\bar{X}^+(l)-\bar{X}(l)}dl}\\&\quad+\int_0^t\begin{split}&\mathbb{E}\bigg[\abs{\bar{X}^+(l-)+e(l)+g\pr{\bar{X}^+(l-)+e(l),\bar{u}(l),y}-\bar{X}(l-)-g\pr{\bar{X}(l-),\bar{u}(l),y}}^2\\&-\abs{\bar{X}^+(l-)-\bar{X}(l-)}^2\bigg]\lambda\mathbb{P}_{C_1}(dy)
\end{split}\\
&\leq \pr{\pp{f}_1+2\lambda\int_{\mathbb{R}}\pr{1+\pp{g}_1\pr{y}}^2\mathbb{P}_{C_1}(dy)}\int_0^t\mathbb{E}\pp{\abs{e(l)}^2}dl\\&\quad+\pr{3\pp{f}_1+\lambda\int_{\mathbb{R}}\pp{1+4\pp{g}_1(y)+2\pp{g}_1^2(y)}\mathbb{P}_{C_1}(dy)}\int_0^t\mathbb{E}\pp{\abs{\pr{\bar{X}^+_1-\bar{X}^+_2}(l)}^2}dl.
\end{split}\end{equation*}
The conclusion is, once again, due to Gronwall's inequality.
\item With the notation $\bar{X}^+:=\bar{X}^{+,\bar{X}_0,\bar{u},e}$, Itô's formula for $\abs{\cdot}^4$ applied on $\pp{0,t}$ to the process $\bar{X}^+$ yields
\begin{equation*}
\begin{split}
&\mathbb{E}\pp{\abs{\bar{X}^{+}(t)}^4}\\
=&\abs{\bar{X}_0}^4+4\int_0^t\mathbb{E}\pp{\abs{\bar{X}^+(l)}^2\scal{f^+\pr{\bar{X}^+(l),\bar{u}(l),e(l)},\bar{X}^+(l)}}dl\\
&+\lambda\int_0^t\mathbb{E}\pp{\abs{\bar{X}^+(l-)+e(l)+g^+\pr{\bar{X}^+(l-),\bar{u}(l),e(l),y}}^4-\abs{\bar{X}^+(l-)}^4}\mathbb{P}_{C_1}(dy)dl\\
\leq &\abs{\bar{X}_0}^4+\frac{\lambda+1}{2}\int_0^t\mathbb{E}\pp{\abs{\bar{X}^+(l)}^4}dl+5^4\max\set{\lambda, 1}\pp{\norm{f}_0^4
+\int_{\mathbb{R}}\norm{g}_0^4(y)\mathbb{P}_{C_1}(dy)}t,
\end{split}
\end{equation*}
where we have used simple inequalities 
\begin{align*}
\begin{cases}
&(\abs{x}+1+g)^4< \frac{3}{2}\abs{x}^4+5^4(1+g)^4;\\ 
&4\abs{x}^3\abs{f}< \frac{1}{2}\abs{x}^4+5^4\abs{f}^4.
\end{cases}
\end{align*}The reader is reminded that $\norm{g}_0$ is actually the bound for $1+g$. The inequality \eqref{CompactEstim} follows again from Gronwall's inequality.
\item Let us fix $q\geq 2$, a predictable control $(\bar{u}:=\pr{u,p},e)$ and initial data $\bar{X}_1(0),\bar{X}_2(0)\in\mathbb{R}^m$. Itô's formula applied to $\abs{\cdot}^q$ and the process $\bar{X}^{+}_1-\bar{X}^{+}_2:=\bar{X}^{+,\bar{X}_1(0),u,p,e}-\bar{X}^{+,\bar{X}_2(0),u,p,e}$ on $\pp{0,t}$ yields
\begin{equation*}
\begin{split}
&\mathbb{E}\pp{\abs{\pr{\bar{X}_1^+-\bar{X}_2^+}(t)}^q}\\
&=\abs{\pr{\bar{X}_1-\bar{X}_2}(0)}^q\\&+q\int_0^t\mathbb{E}\pp{\abs{\bar{X}_1^+-\bar{X}_2^+}^{q-2}\scal{f^+\pr{\bar{X}_1^+(l),\bar{u}(l),e(l)}-f^+\pr{\bar{X}^+_2(l),\bar{u}(l),e(l)},\bar{X}_1^+(l)-\bar{X}^+_2(l)}dl}\\&\begin{split}+\lambda &\int_0^t\mathbb{E}\bigg[\abs{\bar{X}_1^+(l-)+g^+\pr{\bar{X}_1(l-),\bar{u}(l),e(l),y}-\bar{X}_2^+(l-)-g^+\pr{\bar{X}_2^+(l-),\bar{u}(l),e(l),y}}^q\\&\quad -\abs{\bar{X}_1^+(l-)-\bar{X}_2^+(l-)}^q\bigg]\lambda\mathbb{P}_{C_1}(dy).
\end{split}
\end{split}\end{equation*}
Owing to the assumptions on $g$, it follows that  
\begin{equation*}
\mathbb{E}\pp{\abs{\pr{\bar{X}_1^+-\bar{X}_2^+}(t)}^q}\leq \abs{\pr{\bar{X}^+_1-\bar{X}^+_2}(0)}^q+q\pp{f}_1\int_0^t\mathbb{E}\pp{\abs{\pr{\bar{X}^+_1-\bar{X}^+_2}(l)}^q}dl.
\end{equation*}
The conclusion follows, as before, owing to Gronwall's inequality.
Finally, one notes that if $\bar{X}:=X^{\bar{X}_1(0),\bar{u}}$, then \begin{equation*}
\begin{split}
&\mathbb{E}\pp{\abs{\pr{\bar{X}_1^+-\bar{X}}(t)}^q}\\
&=q\mathbb{E}\pp{\int_0^t\abs{\bar{X}_1^+(l)-\bar{X}(l)}^{q-2}\scal{f^+\pr{\bar{X}_1^+(l),\bar{u}(l),e(l)}-f\pr{\bar{X}(l),\bar{u}(l)},\bar{X}_1^+(l)-\bar{X}(l)}dl}\\&\begin{split}&\quad+\int_0^t\mathbb{E}\bigg[\abs{\bar{X}_1^+(l-)+e(l)+g^+\pr{\bar{X}_1(l-),\bar{u}(l),e(l),y}-\bar{X}(l-)-g\pr{\bar{X}(l-),\bar{u}(l),y}}^q\\&\quad \quad\quad\quad\quad-\abs{\bar{X}_1^+(l-)-\bar{X}(l-)}^q\bigg]\lambda\mathbb{P}_{C_1}(dy)\end{split}\\
&\leq q\pp{f}_1\int_0^t\mathbb{E}\pp{\abs{\bar{X}_1^+(l)-\bar{X}(l)}^{q-1}\pr{\abs{e(l)}+\abs{\bar{X}_1^+(l)-\bar{X}^+(l)}}}dl\\&\quad+\lambda\mathbb{E}\pp{\int_0^t\abs{\bar{X}^+_1(l-)+e(l)-\bar{X}(l-)}^qdl}\\
&\leq \pp{(2q-1)\pp{f}_1+\lambda q}\int_0^t\mathbb{E}\pp{\abs{\bar{X}_1^+(l)-\bar{X}(l)}^qdl}+\pr{\pp{f}_1+\lambda q}\int_0^t\mathbb{E}\pp{\abs{e(l)}^q}dl.
\end{split}\end{equation*}where we have used classical inequalities $x^{q-1}a\leq \frac{q-1}{q}x^q+\frac{1}{q}a^q$ and $(x+a)^q\leq q(x^q+a^q)$ for $x,a\in\mathbb{R}_+$. We conclude, as for the other assertions, by using Gronwall's inequality.
\end{enumerate}
\end{proof}
\subsection{Elements of Proof for Theorem \ref{ThGaclco}}
\begin{proof}[Proof of Theorem \ref{ThGaclco}]
 \begin{enumerate}
 \item The first assertion is a direct consequence of \cite[Corollary 8]{goreac_serea_ZubovPDMP} for the extended dynamics $\pr{a,\bar{X}}$. \\
 The only notable differences are the support condition as well as the fourth moment bound. To understand these, one notes that, for an admissible control, $\bar{u}_0$, one gets $a^{a_0,\bar{x}_0,\bar{u}}(t)=a_0e^{-ht}$, for $t\geq 0$, which leads to the support condition for the occupation measure $\gamma=\gamma^{a_0,\bar{x}_0,\bar{u}_0}\in \Gamma(a_0,\bar{x}_0)$. \\
Furthermore, one has \begin{align*}
 \int_{\mathbb{R}^{m+1}\times \bar{U}}\abs{\bar{x}}^4\gamma^{a_0,\bar{x}_0,\bar{u}_0}(da,d\bar{x},d\bar{u})=\mathbb{E}\pp{\int_0^\infty e^{-t}\abs{X^{x_0,\bar{u}_0}(t)}^4dt},
 \end{align*}
 and the estimate on the fourth moment follows from \eqref{CompactEstim} when $h\geq 1\geq \lambda$,  the last inequality being by assumption. The support and moment bounds are easily extended to convex combinations and limits of these, hence providing the characterization in $\Theta$ and they also imply the argument on the Wasserstein space $\mathcal{P}_2\pr{\mathbb{R}^{m+1}}$. Let us emphasize that the nature of our problem only concerns costs independent of $u$. The arguments in this theorem are actually valid for $\mathcal{P}_2\pr{\mathbb{R}^{m+1}\times\bar{U}}$.
 \item The first equality in the second assertion follows from the first assertion for the continuous cost functional functional $(a,\bar{x})\mapsto a^q\phi\pr{\bar{x}}$ by noting that it can actually be considered as a bounded cost functional given the support of the $a$-marginal of measures in $\Theta_1(a_0,\bar{x}_0)$.\\
 
\noindent For the equality between the first term and the third one, one focuses on the $1+qh$-discounted control problem for the basic dynamics $\bar{X}$, take a look at \eqref{HJBq}. The conclusion follows from \cite[Theorem 7]{goreac_serea_ZubovPDMP} (actually, a small variation in which the discount parameter is $1+qh$ instead of $1$). 
\item The lower bound follows from the explicit form of approximating values in \cite[Section 3.3]{goreac_serea_ZubovPDMP} (see the functions $v_g^\varepsilon(\cdot)$ on page 218, then the convolutions $V^\varepsilon$ in Eq. (15) on page 220). For the last two assertions, the reader is kindly invited to also take a look at the proof of Theorem \ref{PropTool1}.
\end{enumerate}
 \end{proof}
 \subsection{Proof of Proposition \ref{PropBI}}
 We provide here the proof of the auxiliary result allowing to obtain the asymptotic behavior of the Hamiltonian in \eqref{HJBeq'}.\\
 
 \begin{proof}[Proof of Proposition \ref{PropBI}]
Let us denote by $\rho$ the continuity modulus of $\phi_q$ for $q\in [2,\infty]$ and by $\rho_\beta$ the continuity modulus of $\beta$.\\
If $\bar{u}\in \bar{U}\pr{\phi_\infty,\nu}$, then $\phi_q\pr{\beta(\bar{u},y)}\leq \phi_\infty\pr{\beta(\bar{u},y)}\leq \nu$. Then, \[\alpha\pr{\bar{u}}+\frac{1}{q}\pr{\frac{\norm{\phi_q\pr{\beta(\bar{u},\cdot)}}_{\mathbb{L}^q\pr{\mathbb{P}_{C_1}}}}{\nu}}^q\mu\leq \alpha\pr{\bar{u}}+\frac{\mu}{q}. \]This yields that the left-hand term cannot exceed the right-hand one. \\
Let us denote by $\bar{\alpha}$ the value of the right-hand term and assume that it is finite and that, for some $\delta>0$,
\[\lim_{q\rightarrow\infty}\min_{\bar{u}\in\bar{U}}\pp{\alpha(\bar{u})+\frac{1}{q}\pr{\frac{\norm{\phi_q\pr{\beta\pr{\bar{u},\cdot}}}_{\mathbb{L}^q\pr{\mathbb{P}_{C_1}}}}{\nu}}^q\mu}<\bar{\alpha}-3\delta.\]
Then, for every $q$ large enough, there exists some $\bar{u}_q$ such that \begin{equation}\label{inequq}\pp{\alpha(\bar{u}_q)+\frac{1}{q}\pr{\frac{\norm{\phi_q\pr{\beta\pr{\bar{u}_q,\cdot}}}_{\mathbb{L}^q\pr{\mathbb{P}_{C_1}}}}{\nu}}^q\mu}<\bar{\alpha}-2\delta.\end{equation} 
By compactness, $\bar{u}_q$ converges (along some subsequence) to some $\bar{u}_\infty\in\bar{U}$.  Furthermore \[\norm{\phi_q\pr{\beta\pr{\bar{u}_q,\cdot}}}_{\mathbb{L}^q\pr{\mathbb{P}_{C_1}}}\leq \nu \pr{\frac{q}{\mu}\pr{\bar{\alpha}-2\delta-\alpha\pr{\bar{u}_q}}}^{\frac{1}{q}}.\]
The reader is reminded that \begin{align*}
\abs{\norm{\phi_q\pr{\beta\pr{\bar{u}_q,\cdot}}}_{\mathbb{L}^q\pr{\mathbb{P}_{C_1}}}-\norm{\phi_q\pr{\beta\pr{\bar{u}_\infty,\cdot}}}_{\mathbb{L}^q\pr{\mathbb{P}_{C_1}}}}
&\leq\norm{\phi_q\pr{\beta\pr{\bar{u}_q,\cdot}}-\phi_q\pr{\beta\pr{\bar{u}_\infty,\cdot}}}_{\mathbb{L}^q\pr{\mathbb{P}_{C_1}}}\\
&\leq \rho\pr{\rho_\beta\pr{\abs{\bar{u}_q-\bar{u}_\infty}}}, 
\end{align*}such that
$\norm{\phi_q\pr{\beta\pr{\bar{u}_\infty,\cdot}}}_{\mathbb{L}^q\pr{\mathbb{P}_{C_1}}}\leq \nu \pr{q\frac{\bar{\alpha}}{\mu}}^{\frac{1}{q}}+ \rho\pr{\rho_\beta\pr{\abs{\bar{u}_q-\bar{u}_\infty}}}.$Taking the limit as $q\rightarrow\infty$, it follows that $\bar{u}_\infty\in \bar{U}(\phi_\infty,\nu)$. Going back to the initial inequality satisfied by $\bar{u}_q$ Eq. \eqref{inequq}, it follows, from the continuity of $\alpha$, that $\alpha\pr{\bar{u}_\infty}<\bar{\alpha}$ which is a contradiction with the definition of $\bar{\alpha}$.\\
When $\bar{\alpha}=+\infty$, it follows that $\bar{U}(\phi_\infty,\nu)=\emptyset$. By assuming that the left-hand limit is finite, we can apply the same kind of reasoning to show the existence of $\bar{u}_\infty\in  \bar{U}(\phi_\infty,\nu)$, thus contradicting the emptiness. It follows that, in this case too, the equality holds true.
\end{proof}
\subsection{Proofs for Section \ref{SecClassSol}}
\begin{proof}[Sketch of the Proof of Proposition \ref{PropU}]
\begin{enumerate}
\item The first assertion is obvious.
\item For the second assertion, let $\omega$ be the the common (non-decreasing) continuity modulus of $(\psi_q)_{q\in\mathbb{R}_+\cup\set{\infty}}$; recall that, by convention, $\psi_\infty=\psi$, and \[\eta(q_0)=2\sup_{q\geq q_0}\sup_{x'\in B_{1+\norm{g}_0}(x)}\abs{\psi_q(x')-\psi(x'))}.\] Then
\begin{align*}
 \psi_{q'}(x+g(x,u,y))\leq \psi_q(x'+g(x',u,y))+\omega \pr{(1+\pp{g}_1)\abs{x-x'}}+\eta(q_0),
\end{align*}and the conclusion follows by picking $\varepsilon=r'-r$, $q_0$ large enough such that $\eta(q_0)\leq \frac{\varepsilon}{2}$ and $\delta$ small enough such that $\omega \pr{(1+\pp{g}_1)\delta}\leq \frac{\varepsilon}{2}$.
\item The assertions 3. (a) and (b) immediately follow from the control sets inclusion. For (c), by fixing $r_0<r$, one easily shows, due to the previous assertions, that $H^*(\psi,x,r,p)\leq H(\psi,x,r_0,p)$ and, as a consequence, $H^*(\psi,x,r,p)\leq H(\psi,x,r-,p)$. The reverse inequality follows from the definition of $H^*$. 
\end{enumerate}
\end{proof}\\

Finally, let us come to the viscosity characterization of $V$.\\

\begin{proof}[Elements of Proof for Theorem \ref{ThClassSol}]
Let $\phi$ be a $C^1$ function s.t. $V-\phi$ admits a strict local minimum at $x$. We further assume $V(x)=\phi(x)$. Since $V_q$ converges uniformly to $V$, it follows that, for every $0<\eta<1$ small enough, $\inf_{x'\in\partial B_\eta(x)}(V_q-\phi)(x')>0$. On the closed set $B_\eta(x)$, the minimum is attained at some $x_q$. \\
\Db{The reader is invited to note that  $\lim_{q\rightarrow\infty}x_q=x$, along some subsequence, still indexed by $q$ for simplicity. Indeed, by compactness of $B_\eta(x)$, one can obtain a subsequence $x_q$ converging to some $y\in B_\eta(x)$. One has $V_q(x_q)-\phi(x_q)\leq V_q(x)-\phi(x)$, and, due to the uniform convergence, and by passing to the limit along this subsequence, $V(y)-\phi(y)\leq V(x)-\phi(x)$. It follows that $y=x$. }\\
The viscosity condition for $V_q$ with the test function $\phi$ and written at $x_q$ yields
\begin{align*}
-\frac{1+hq+\lambda}{q}\phi(x_q)+\underset{\bar{u}\in \bar{U}}{\inf}\Big\{&\scal{f(x_q,\bar{u}),\nabla\phi(x_q)}+\frac{1}{q}\pr{\frac{L(x_q)}{\phi(x_q)}}^q\phi(x_q)\\&+\frac{\lambda}{q}\int_{\mathbb{R}}\pp{\frac{\phi(x_q+g(x_q,\bar{u},y))}{\phi(x_q)}}^q\mathbb{P}_{C_1}(dy)\phi(x_q)\Big\}\leq 0.
\end{align*}
In particular, it follows that, for some constant $k>0$ (independent of $q$), 
\begin{align*}
\begin{cases}
\frac{1}{q}\pr{\frac{L(x_q)}{\phi(x_q)}}^q\phi(x_q)\leq k,\\
{\inf}_{\bar{u}\in \bar{U}}\frac{\lambda}{q}\int_{\mathbb{R}}\pp{\frac{\phi(x_q+g(x_q,\bar{u},y))}{\phi(x_q)}}^q\mathbb{P}_{C_1}(dy)\leq k.
\end{cases}
\end{align*}
As $q\rightarrow\infty$, we have $L(x)\leq \phi(x)$. Assuming this fails to hold, for some $\varepsilon>0$ and $q$ large enough, \[L(x_q)\geq L(x)(1-\varepsilon)\geq \phi(x)(1+2\varepsilon)\geq \phi(x_q)(1+\varepsilon), \] which leads to a contradiction, since $\frac{(1+\varepsilon)^q}{q}$ diverges to $+\infty$, while $\phi(x_q)\geq (1-\varepsilon)\phi(x)=(1-\varepsilon)V(x)$ is strictly positive. 
\\
Furthermore, for every $\varepsilon>0$ small enough, then every $q\geq q_\varepsilon$ large enough,
\begin{align*}
0\geq &-\varepsilon-\frac{1+hq+\lambda}{q}\phi(x)\\&+\underset{\bar{u}\in \bar{U}}{\inf}\Big\{\scal{f(x,\bar{u}),\nabla\phi(x)}+\frac{\lambda}{q}\int_{\mathbb{R}}\pp{\frac{\phi(x+g(x,\bar{u},y))-\varepsilon}{\phi(x)+\varepsilon}}^q\mathbb{P}_{C_1}(dy)(\phi(x)-\varepsilon)\Big\}.
\end{align*}
Letting $q\rightarrow\infty$,  then $\varepsilon\rightarrow 0+$, one gets $\phi(x)\geq L(x)$ and
\[-h\phi(x)+H(\phi,x,\phi(x)+,\nabla\phi(x))\leq 0.\]The subsolution condition is quite similar.\\
The proof of the comparison result is quasi-identical to \cite[Proposition 19]{GHS_2021}. We are here in the infinite horizon case, there is no time modification and the contradiction comes in a standard way from the $-h\psi$ term in the equation \eqref{HJB}.
\end{proof}
\subsection{Sketch of the Proof of Proposition \ref{PropPremium}}
\begin{proof}
We begin with the following estimates on the edges.  For simplicity, we drop the upper-script describing the dependence on the initial data and the control parameters. To this purpose, we assume that one uses a piecewise constant control couple $(u,p)$ \Db{and consider $t>0$ small enough such that this constant policy is considered on $[0,t]$}.
We have the following straightforward developments, with the notations $\bar{X}=(S,I,X)$ and the coefficients $f$ and $g$ as described in Section \ref{SubsfgIns}.
\[\bar{X}(t)= \bar{X}(0)+o\pr{\sqrt{t}},\ \textnormal{in }\mathbb{L}^2 \textnormal{ (hence in }\mathbb{L}^1);\]
\[\bar{X}(t)= \bar{X}(0)+\int_0^t f(\bar{X}(0),u,p)dl+g(\bar{X}(0),\Db{u,p},C_1)\mathbf{1}_{\tau_1\leq t}+o(t),\textnormal{ in }\mathbb{L}^1;\]
\[c(S(0),I(0),u,p)=\lim_{t\rightarrow 0+}\frac{1}{t}\int_0^t c(S(l),I(l),u,p)dl,\ \textnormal{ in }\mathbb{L}^1;\]
\begin{align*}&\mathbb{E}\pp{\int_0^t\sum_{1\leq j\leq n}\int_{\mathbb{R}_+}\pr{\frac{I^{0,u,p}(t-)}{p_j(t)}-I_j^{u,p}(t-)}^+yN_p(ds,dy)}\\&=\sum_{1\leq j\leq n}\mathbb{E}\pp{\pr{I_j(\tau_1)-I_j(\tau_1-)}\mathbf{1}_{\tau_1\leq t}}\mathbb{E}\pp{C_1}+o(t)=\sum_{1\leq j\leq n}\pr{\frac{\sum_{1\leq \Db{k} \leq n}I_{\Db{k}}(0)}{np_j}-I_j(0)}^+\lambda t\mathbb{E}\pp{C_1}+o(t), 
\end{align*}thus leading to the statement.
\end{proof}
\bibliographystyle{abbrv}      
\bibliography{bibl_2024_fin} 
\end{document}